\documentclass[10pt]{amsart}

\usepackage[colorlinks]{hyperref}
\usepackage{color,graphicx,shortvrb}
\usepackage[latin 1]{inputenc}

\usepackage[active]{srcltx} 

\usepackage{enumerate}
\usepackage{amssymb}

\newtheorem{theorem}{Theorem}[section]

\newtheorem{corollary}[theorem]{Corollary}

\newtheorem{lemma}[theorem]{Lemma}

\newtheorem{proposition}[theorem]{Proposition}
\newtheorem{remark}[theorem]{Remark}

\def\11{\textbf{$1$}}

\DeclareGraphicsExtensions{.jpg,.pdf,.png,.eps}

\begin{document}

\title{Tingley's problem for $p$-Schatten von Neumann classes}

\author[F.J. Fern\'{a}ndez-Polo]{Francisco J. Fern\'{a}ndez-Polo}

\author[E. Jord\'a]{Enrique Jord\'a}

\address{Escuela Polit\'ecnica Superior de Alcoy, IUMPA, Universitat Polit\'ecnica de Valencia, Plaza Ferr\'andiz y Carbonell 1, 03801 Alcoy, Spain}
\email{ejorda@mat.upv.es}

\author[A.M. Peralta]{Antonio M. Peralta}

\address[F.J. Fern\'{a}ndez-Polo, A.M. Peralta]{Departamento de An{\'a}lisis Matem{\'a}tico, Facultad de
Ciencias, Universidad de Granada, 18071 Granada, Spain.}
\email{pacopolo@ugr.es}
\email{aperalta@ugr.es}




\date{}

\begin{abstract} Let $H$ and $H'$ be a complex Hilbert spaces. For $p\in(1, \infty)\backslash\{2\}$ we consider the Banach space $C_p(H)$ of all $p$-Schatten von Neumann operators, whose unit sphere is denoted by $S(C_p(H))$. We prove that every surjective isometry $\Delta: S(C_p(H))\to S(C_p(H'))$ can be extended to a complex linear or to a conjugate linear surjective isometry $T:C_p(H)\to C_p(H')$.
\end{abstract}

\maketitle
\thispagestyle{empty}

\section{Introduction}

Tingley's problem has lured a multitude of researchers interested in determining if this problem admits a positive solution in the general setting or in some particular classes of Banach spaces. Given a Banach space $X$, the symbol $S(X)$ will stand for the unit sphere of $X$. D. Tingley proved in \cite{Ting1987} that, for finite dimensional Banach spaces $X$ and $Y$, a surjective isometry $\Delta :S(X)\to S(Y)$ satisfies $\Delta (-x) = - \Delta(x),$ for every $x\in S(X)$. Tingley's theorem gives rise to the so-called Tingley's problem, which can be considered as a generalization of the Mazur-Ulam theorem. The problem studied nowadays can be settled in the following terms: Suppose $\Delta : S(X) \to S(Y)$ is a surjective isometry between the unit spheres of two arbitrary Banach spaces $X$ and $Y$. Does $\Delta$ extend to a real linear isometry from $X$ onto $Y$? The question remains open when $X$ and $Y$ are arbitrary 2-dimensional Banach spaces.\smallskip

The achievements obtained during the thirty years of history around Tingley's problem can be hardly resumed in one or two paragraphs. Part of the most relevant results to place Tingley's problem in its true historical perspective for our purposes include positive answers for surjective isometries between the unit spheres of $\ell_p (\Gamma)$ spaces with $1\leq p \leq \infty$ \cite{Ding2002,Di:p,Di:8,Di:1}. In the setting of commutative structures positive solutions to Tingley's problem have been also established for spaces of measurable functions of the form $L^{p}(\Omega, \Sigma, \mu),$ where $(\Omega, \Sigma, \mu)$ is a $\sigma$-finite measure space and $1\leq p\leq \infty$ \cite{Ta:8, Ta:1, Ta:p}, and spaces of continuous functions \cite{Wang}. Some of these spaces actually satisfy a stronger property, the \emph{Mazur-Ulam property}. We briefly recall that a Banach space $X$ satisfies the Mazur-Ulam property if for every Banach space $Y$, Tingley's problem admits a positive solution for every surjective isometry $\Delta : S(X) \to S(Y)$. Real sequence spaces like $c (\Gamma,\mathbb{R})$, $c_{0} (\Gamma,\mathbb{R})$, and $\ell_{\infty} (\Gamma,\mathbb{R})$ satisfy the Mazur-Ulam property. The spaces $C(K,\mathbb{R})$, $L^{p}((\Omega, \Sigma, \mu), \mathbb{R})$ also have this property (see \cite{Liu2007, FangWang06, Ta:1,Ta:8} and \cite{Ta:p}). The results in the recent papers \cite{JVMorPeRa2017,Pe2017} show that the spaces of complex sequences $c_0(\Gamma)$ and $\ell_{\infty}(\Gamma)$ also satisfy the Mazur-Ulam property.\smallskip

The commutative triplet $(c_0,c_0^*=\ell_1,\ell_1^*= \ell_{\infty})$ admits a non-commutative analogue of the form $(K(H),K(H)^*=C_1(H),C_1(H)^*=B(H)),$ where $H$ is a complex Hilbert space and $K(H),$ $C_1(H),$ and $B(H)$ are the spaces of compact, trace class, and bounded linear operators on $H$, respectively. R. Tanaka gave a positive solution to Tingley's problem for surjective isometries between the unit spheres of two finite von Neumann algebras \cite{Tan2016,Tan2017,Tan2017b}. R. Tanaka and the third author of this note found a complete solution to Tingley's problem for surjective isometries between the unit spheres of two $K(H)$ spaces or between two compact C$^*$-algebras (see \cite{PeTan16}). J. Garc{\'e}s, I. Villanueva in collaboration with the first and third author of this note solved the problem for the space of trace class operators \cite{FerGarPeVill17}. Additional solutions to Tingley's problem for $B(H)$ spaces, atomic von Neumann algebras and JBW$^*$-triples, and general von Neumann algebras are due to the first and third author of this note in \cite{FerPe17,FerPe17b,FerPe17c}, and \cite{FerPe17d}. The most recent achievement in this line is a result by M. Mori, which establishes that a surjective isometry between the unit spheres of two von Neumann algebra preduals admits a unique extension to a surjective real linear isometry between the corresponding spaces \cite{Mori2017}. We refer to the surveys \cite{Ding2009,Pe2018,YangZhao2014} for a detailed overview on Tingley's problem.\smallskip

During a talk presented by the third author of this note in the Conference on Non-Linear Functional Analysis held at the Universitat Polit\`{e}cnica de Valencia (Spain) in 2017, Professor Andreas Defant asked whether Tingley's problem admits a positive solution for the spaces, $C_p(H),$ of $p$-Schatten von Neumann operators on a complex Hilbert space $H$. By the non-commutative Clarkson-McCarthy inequalities, the space $C_p(H)$ is uniformly convex for every $1<p<\infty$ {\rm(}compare \cite{McCarthy67}{\rm)}, and hence strictly convex. Every point in $S(C_p(H))$ is an extreme point of the closed unit ball of $C_p(H)$. In particular, the unit sphere of $C_p(H)$ contains no segments. This is just one of the reasons due to which the usual techniques applied in the different solutions to Tingley's problem presented in the forerunners surveyed above are useless in this particular setting. No answer to Tingley's problem has been established for this non-commutative generalizations of $\ell_p$ spaces, this is the aim of this paper.\smallskip

In this note we present a complete solution to Tingley's problem for surjective isometries between the unit spheres of two $p$-Schatten von Neumann spaces for every $p\in (1,\infty)$ (see Theorem \ref{t Tingley's for Cp with 2<p}). In order to present our results, let the symbol $\mathcal{U}_{min} (H)$ stand for the set of all minimal partial isometries in $C_p(H)$. In our arguments we first establish that, given two complex Hilbert spaces $H$ and $H'$, $1< p <\infty$, $p\neq 2$, and a surjective isometry  $\Delta: S(C_p(H))\to S(C_p(H'))$, then $\Delta$ maps minimal partial isometries in $C_p(H)$ to minimal partial isometries in $C_p(H')$, that is, $\Delta (\mathcal{U}_{min} (H))= \mathcal{U}_{min} (H')$ (see Proposition \ref{p minimal partial isometries in arbitrary dimension}). Consequently, the restriction $\Delta|_{_{\mathcal{U}_{min} (H)}}: \mathcal{U}_{min} (H)\to \mathcal{U}_{min} (H')$ is a surjective isometry (see Corollary \ref{c minimal partial isometries in arbitrary dimension}). Several technical results are established to determine that $\Delta$ or a certain composition of $\Delta$ with a conjugation preserves the transition probabilities between elements in $\mathcal{U}_{min} (H)$. These technical results provide the appropriate conditions to apply a generalization of Wigner's theorem established by L. Moln{\'a}r in \cite{Mol2002}.\smallskip

Our strategy is completed with a generalization of a result obtained by G. Nagy. Let $S(C_p(H)^+)$ denote the unit sphere of positive operators in $C_p(H)$. In \cite{Nag2013} G. Nagy proves that every surjective isometry $\Delta : S(C_p(H)^+)\to S(C_p(H)^+)$ admits a unique extension to a surjective complex linear or a conjugate linear isometry on $C_p(H)$. In \cite[Lemma in page 3]{Nag2013} Nagy proves the following: Let $H$ be a finite dimensional complex Hilbert space, $1<p< \infty$ and $1\leq \gamma$. If $a,b\in S(C_p(H)^+)$ satisfy $\|a-\gamma p\|_p = \|b-\gamma p\|_p$ for every minimal projection $p\in C_p(H)^+$, then $a = b$. Our generalization is another identity principle established in Proposition \ref{p a la Nagy}, where we prove that if $a,b\in S(C_p(H))$ satisfy $\|a-\gamma e\|_p = \|b-\gamma e\|_p$ for every $e\in\mathcal{U}_{min} (H)$, then $a = b$. These are the main tools leading to our main result.

\section{The results}

Let $H$ be a complex Hilbert space. We are interested in different subclasses of the space
$K(H)$ of all compact operators on $H$. We briefly recall the basic terminology.
For each compact operator $a$, the operator $a^* a$ lies in $K(H)$ and admits a
unique square root $|a| = (a^* a)^{\frac12}$. The \emph{singular values}
of the operator $a$
are precisely the eigenvalues of $|a|$ arranged in decreasing order
and repeated according
to multiplicity. Since $|a|$ belong to $K(H)$, only an at most countable number of
its eigenvalues are greater than zero.
Accordingly to the standard terminology, we usually write $\sigma_n (a)$ for the $n$-th
singular value of $a$. It is well known that $(\sigma_n (a))_n\to 0$.\smallskip

Given $1\leq p<\infty$ the class $C_p(H)$ is the set of all $a$ in $K(H)$
such that $$\hbox{tr}(|a|^p) = \|a\|_p^p := \left(\sum_{n=1}^{\infty} |\sigma_n (a)|^p \right) <\infty.$$
We set $\|a\|_{\infty} =\|a\|$, where the latter stands for the operator norm of $a$. The set $C_p(H)$ is a two-sided ideal
in the space $B(H)$ of all bounded linear operators on $H$, and $(C_p(H), \|.\|_p)$ is a Banach algebra. $C_2(H)$ is the class of Hilbert-Schmidt operators, $C_1(H)$ is the space of trace class operators, and $C_p(H)$ is the space of $p$-Schatten von Neumann operators. If tr$(.)$ denotes the usual trace on $B(H)$ and $a\in K(H)$, we know that $a\in C_1(H)$ if, and only if, tr$(|a|)<\infty$ and $\|a\|_1 = \hbox{tr} (|a|)$.  It is further known that the predual of $B(H),$ and the dual of $K(H),$ both can be identified with $C_1(H)$ under the isometric linear mapping $a\mapsto  \varphi_a$, where $\varphi_a (x) := \hbox{tr} (a x)$ ($a\in C_1(H), x\in B(H)$).\smallskip

An element $e\in B(H)$ is a partial isometry if $ee^*$ (equivalently, $e^*e$) is a projection, or equivalently, if and only if $ee^*e=e$.
It is known that every element $a$ in $C_p(H)$ can be written as a (possibly finite) sum
\begin{equation}\label{eq spectral resolution in trace class}
a= \sum_{n=1}^{\infty} \lambda_n \eta_n \otimes \xi_n,
\end{equation} where $(\lambda_n)\subset \mathbb{R}_0^{+}$,
$(\xi_n)$, $(\eta_n)$ are orthonormal systems in $H$,  and
$\displaystyle \|a\|_p^p=\sum_{n=1}^{\infty} \lambda_n^p$.\smallskip

Given an element $a$ in $C_p(H)$ written in the form given in \eqref{eq spectral resolution in trace class}, the element $$s(a) = \sum_{n=1}^{\infty} \eta_n \otimes \xi_n,$$ is a partial isometry in $B(H)$ (called the support partial isometry of $a$ in $B(H)$).\smallskip

We refer to \cite{McCarthy67}, \cite[Chapter III]{GohbergKrein}, \cite[\S 9]{DunSchw63}, \cite[Chapter II]{Tak} and \cite[\S 1.15]{S} for the basic results and references on $C_p(H)$ spaces.\smallskip

The non-commutative Clarkson-McCarthy inequalities state that the formulae
\begin{equation}\label{Clarkson-McCarthy ineq 0pleq2} 2^{p-1} \left( \|a\|_p^p + \|b\|_p^p\right) \leq \|a+ b\|_p^p+ \|a-b\|_p^p \leq 2 \left( \|a\|_p^p + \|b\|_p^p \right) \ \ (0<p\leq 2);
\end{equation}
\begin{equation}\label{Clarkson-McCarthy ineq 2p infty} 2 \left( \|a\|_p^p + \|b\|_p^p \right) \leq \|a+ b\|_p^p+ \|a-b\|_p^p \leq 2^{p-1} \left( \|a\|_p^p + \|b\|_p^p \right) \ \ (2\leq p < \infty),
\end{equation} hold for all $a,$ $b$ in $C_p(H)$  (cf. \cite[Theorem 2.7]{McCarthy67}). It is further known that if $p\in [1,\infty)\backslash\{2\}$,  equality $$\|a+ b\|_p^p+ \|a-b\|_p^p = 2 \left( \|a\|_p^p + \|b\|_p^p \right)$$ holds in \eqref{Clarkson-McCarthy ineq 0pleq2} or in \eqref{Clarkson-McCarthy ineq 2p infty} if and only if $(a^* a) (b^* b)=0$. Since $\|c\|_p= \|c^*\|_p$ for every $c\in C_p(H)$ (see \cite[Theorem 1.3]{McCarthy67}), it follows that the previous equalities hold if and only if $a$ and $b$ are orthogonal as elements in $B(H)$ ($a\perp b$ in short), that is $a b^* =0 =b^* a$, or in other words $s(a)$ and $s(b)$ are orthogonal partial isometries in $B(H)$ (i.e. $s(a) s(b)^*=0$ and $s(b)^* s(a)=0$). Consequently, for $1\leq p <\infty$, $p\neq 2$, if we fix $a,b\in S(C_p(H)),$ we can conclude that \begin{equation}\label{eq orthogonality in Cp} \|a\pm b\|_p^p = 2 \Longleftrightarrow a\perp b \hbox{ (in $C_p(H)$)} \Longleftrightarrow s(a)\perp s(b) \hbox{ (in $B(H)$).}
\end{equation}

Let us recall a technical result due to G.G. Ding (see \cite{Ding2002}).

\begin{lemma}\label{l Ding strictly convex}\cite[Lemma 2.1]{Ding2002} Let $X$ and $Y$ be normed spaces.
Suppose $X$ is strictly convex, and $\Delta: S(X) \to S(Y)$ is a mapping.
If $-\Delta(S(X)) \subset \Delta(S(X))$ and
$$\| \Delta(x_1) - \Delta(x_2)\| \leq  \| x_1 - x_2\|\ \ \forall x_1, x_2 \in S(X),$$
then $\Delta$ is one-to-one, and $\Delta(-x) = -\Delta(x)$ for all $x\in S(X)$.$\hfill\Box$
\end{lemma}

We can now deduce a result similar to that obtained by Tingley in \cite{Ting1987} for surjective isometries from the unit sphere of $C_p(H)$ ($1<p<\infty$) onto the unit sphere of another normed space.

\begin{remark}\label{r strictly convex commutes with -1} It follows from the non-commutative
Clarkson-McCarthy inequalities that $C_p(H)$ is uniformly convex for
every $1<p<\infty$ {\rm(}compare \cite{McCarthy67}{\rm)}. It is known that every uniformly
convex space is strictly convex. Thus, given a normed space $Y,$ and a surjective isometry
$\Delta: S(C_p(H))\to S(Y)$, it follows from Lemma \ref{l Ding strictly convex} that $\Delta(-x) =-\Delta(x),$
for every $x\in S(C_p(H))$.
\end{remark}

%

Suppose $\{\xi_i\}_{i\in I}$ is an orthonormal basis of $H$. The elements in the set $\{\xi_i\otimes \xi_i: i \in I \}$ are mutually orthogonal in $C_p (H)$. Actually, the dimension of $H$ is precisely the cardinal of the biggest set of mutually orthogonal elements in $C_p(H)$.\smallskip

We can state now a non-commutative version of \cite[Lemma 3]{Di:1},
\cite[Lemma 3]{Di:p}, \cite[Lemma 3.7]{Ta:p}, and an extension of \cite[Lemma 2.2]{FerGarPeVill17} for $1<p<\infty$, $p\neq 2$.

\begin{lemma}\label{l preservation of orthogonality} Let $H$ and $H'$ be complex Hilbert spaces,
let $1\leq p <\infty$, $p\neq 2$, and let $\Delta: S(C_p(H))\to S(C_p(H'))$ be a surjective isometry.
Then $\Delta$ preserves orthogonal elements in both directions, that is, $a\perp b$ in $S(C_p(H))$ if and only if $\Delta(a)\perp \Delta(b)$ in $S(C_p(H'))$. In particular, dim$(H) = \hbox{dim} (H')$.
\end{lemma}

\begin{proof} Take $a,b$ in $S(C_p(H))$. We have already commented that $a\perp b$
if and only if $\|a\pm b\|_p^p = 2$ (compare \eqref{eq orthogonality in Cp}).
Since $\Delta$ is an isometry we deduce that $$\|\Delta(a)-\Delta(b)\|_p^p=2.$$

Remark \ref{r strictly convex commutes with -1} implies that $\Delta(-x) = -\Delta(x)$ for every $x\in  S(C_p(H))$,
and hence $$\| \Delta(a) + \Delta(b)\|_p^p = \| \Delta(a) - \Delta(-b)\|_p^p = \| a + b\|_p^p =2.$$
Therefore $\|\Delta(a) \pm \Delta(b)\|_p^p=2$,
and hence $\Delta(a) \perp \Delta(b)$.\smallskip

The final conclusion follows from the comments preceding this lemma.
\end{proof}

We recall that a partial isometry is called \emph{minimal} if it is a rank one partial isometry. The set of all minimal partial isometries in $C_p(H)$ will be denoted by $\mathcal{U}_{min} (H)$. Along the paper, $\mathbb{T}$ will stand for the unit sphere of $\mathbb{C}$.

\begin{proposition}\label{p minimal partial isometries in arbitrary dimension} Let $H$ and $H'$ be complex Hilbert spaces, let $1< p <\infty$, $p\neq 2$, and let $\Delta: S(C_p(H))\to S(C_p(H'))$ be a surjective isometry. Then the following statements hold:\begin{enumerate}[$(a)$]\item $\Delta$ maps minimal partial isometries in $C_p(H)$ to minimal partial isometries in $C_p(H')$. Furthermore, $\Delta (\mathcal{U}_{min} (H))= \mathcal{U}_{min} (H')$;
\item For each $e_0\in \mathcal{U}_{min} (H)$ we have $\Delta(\lambda e_0 ) = \lambda \Delta(e_0)$ or $\Delta(\lambda e_0 ) = \overline{\lambda} \Delta(e_0),$ for all $\lambda\in \mathbb{T}$;
\item For each $e_0\in \mathcal{U}_{min} (H)$ if  $\Delta(\mu e_0 ) = \mu \Delta(e_0)$ {\rm(}respectively, $\Delta(\mu e_0 ) = \overline{\mu} \Delta(e_0)${\rm)} for some $\mu\in\mathbb{T}\backslash\{\pm 1\}$, then  $\Delta(\lambda e_0 ) = \lambda \Delta(e_0)$ {\rm(}respectively, $\Delta(\lambda e_0 ) = \overline{\lambda} \Delta(e_0)${\rm),} for every $\lambda\in \mathbb{C}$ with $|\lambda|=1$;
\item Let $\{\eta_j: j\in J\}$ and $\{\xi_j: j\in J\}$ be orthonormal bases of $H$. Then there exist orthonormal bases $\{\widetilde{\eta}_j: j\in J\}$ and $\{\widetilde{\xi}_j: j\in J\}$ of $H'$ such that $\Delta (\eta_j\otimes \xi_j) = \widetilde{\eta}_j\otimes \widetilde{\xi}_j$ for all $j\in J$.
\end{enumerate}

\end{proposition}

\begin{proof} $(a)$ Let $e = \xi_0\otimes \eta_0$ be a minimal partial isometry in $C_p(H)$.\smallskip

%
%

Let us pick a family $\{z_i : i \in I\}\subset S(C_p(H))$ such that $$\left(\bigcap_{i\in I} \{z_i \}^{\perp}\right)\cap S(C_p(H)) = \mathbb{T} e.$$ If $\Delta(e)$ is not a minimal partial isometry, we can find a (possibly finite) collection of mutually orthogonal minimal partial isometries $\{v_n : n\in J\}\subseteq C_p(H')$, with $\sharp J \geq 2$, and a sequence $(\mu_n)_n\in \ell_p$, with $\mu_j\neq 0$ for all $j\in J$, such that $\displaystyle \Delta (e) = \sum_{n=1}^{\infty} \mu_n v_n$. By Lemma \ref{l preservation of orthogonality} $\Delta (e) \perp \Delta(z_i)$, for all $i\in I$. Let us take $j_1\neq j_2$ in $J$. We observe that $v_{{j_1}}\perp v_{j_2}$, and $v_{{j_1}}, v_{j_2}\perp \Delta(z_i)$, for all $i\in I$. By applying Lemma \ref{l preservation of orthogonality} to $\Delta^{-1}$ we conclude that $\Delta^{-1} (v_{j_1})\perp \Delta^{-1} (v_{j_2})$ and $\Delta^{-1} (v_{j_1}), \Delta^{-1} (v_{j_2}) \perp z_i,$ for all $i\in I$, which is impossible.\smallskip

$(b)$ Let us pick a minimal partial isometry $e= \eta_0\otimes \xi_0$ in $C_p(H)$ and $\lambda\in \mathbb{T}$. We can find orthonormal bases $\{\eta_0\}\cup\{\eta_j: j\in J\}$ and $\{\xi_0\}\cup \{\xi_j: j\in J\}$ in $H$. Clearly, the element $e$ belongs to the set $\left(\bigcap_{j\in J} \{e_j \}^{\perp}\right)\cap S(C_p(H)) = \mathbb{T} (\eta_{0}\otimes \xi_{0}),$ where $e_j = \eta_j\otimes \xi_j$. We deduce from Lemma \ref{l preservation of orthogonality} that $$\lambda \Delta(e)\in \left(\bigcap_{j\in J} \{\Delta(e_j) \}^{\perp}\right)\cap S(C_p(H')) = \Delta(\mathbb{T} (\eta_{0}\otimes \xi_{0})) = \Delta(\mathbb{T}e),$$ and thus, there exists $\mu\in \mathbb{T}$ satisfying $\Delta(\mu e) = \lambda  \Delta (e)$. Now, by Remark \ref{r strictly convex commutes with -1} we get $$|\mu \pm 1|= \| \mu e \pm e\| = \|\Delta(\mu e) \pm \Delta(e)  \| = \| \lambda  \Delta(e) \pm \Delta(e)\| = |\lambda \pm 1| , $$ which assures that $\mu \in\{\lambda,\overline{\lambda}\}$ as desired.\smallskip

$(c)$ Suppose now that $e_0\in \mathcal{U}_{min} (H)$ with $\Delta(\mu e_0 ) = \mu \Delta(e_0)$ {\rm(}respectively, $\Delta(\mu e_0 ) = \overline{\mu} \Delta(e_0)${\rm)} for some $\mu\in \mathbb{T}\backslash\{\pm 1\}$. Take $\lambda\in \mathbb{T}\backslash\{\pm1\}$. By $(b)$ we have $\Delta(\lambda e_0 ) = \lambda \Delta(e_0)$ or $\Delta(\lambda e_0 ) = \overline{\lambda} \Delta(e_0)$. Since $$  \|\Delta(\lambda e_0 ) - \mu \Delta( e_0 )\|  =  \|\Delta(\lambda e_0 )- \Delta(\mu e_0 )\|  =   \| \lambda e_0 - \mu e_0 \|  =|\lambda -\mu| $$ $\hbox{(respectively, }  \|\Delta(\lambda e_0 ) - \mu \Delta( e_0 )\|  $ $=  \|\Delta(\lambda e_0 ) - \Delta(\overline{\mu} e_0 )\|  = $ $ \| \lambda e_0 -\overline{\mu} e_0 \|  =|\lambda - \overline{\mu}| ),$ it can be easily deduced that $\Delta(\lambda e_0 ) = \lambda \Delta(e_0)$ {\rm(}respectively, $\Delta(\lambda e_0 ) = \overline{\lambda} \Delta(e_0)${\rm)}. \smallskip

$(d)$ In the hypothesis of our statement, it follows from $(a)$ that, for each $j\in J$, the element $\Delta (\eta_j\otimes \xi_j)$ must be a minimal partial isometry. Lemma \ref{l preservation of orthogonality} implies that $\Delta (\eta_j\otimes \xi_j)\perp \Delta (\eta_k\otimes \xi_k)$ for all $j\neq k$ in $J$. We can therefore find orthonormal systems $\{\widetilde{\eta}_j: j\in J\}$ and $\{\widetilde{\xi}_j: j\in J\}$ in $H'$ such that $\Delta (\eta_j\otimes \xi_j) = \widetilde{\eta}_j\otimes \widetilde{\xi}_j$ for all $j\in J$. Fix $j_0$ in $J$. If $\{\widetilde{\xi}_j: j\in J\}$ is not a basis in $H'$, then there exists a norm-one element $\widetilde{\xi}_0$ in $H'$ such that $\widetilde{\xi}_0\perp  \widetilde{\xi}_j$ for all $j\in J$. By $(a)$, the element $\Delta^{-1}(\widetilde{\eta}_{j_0}\otimes \widetilde{\xi}_0) = \eta\otimes \xi$ is a minimal partial isometry in $B(H)$. Since $\widetilde{\eta}_{j_0}\otimes \widetilde{\xi}_0$ is orthogonal to $\widetilde{\eta}_{j}\otimes \widetilde{\xi}_j$ for every $j\neq j_0$, Lemma \ref{l preservation of orthogonality} assures  that $\Delta^{-1}(\widetilde{\eta}_{j_0}\otimes \widetilde{\xi}_0) = \eta\otimes \xi$ is orthogonal to  $\eta_j\otimes \xi_j$ for every $j\neq j_0$. Since $\{{\eta}_j: j\in J\}$ and $\{{\xi}_j: j\in J\}$ are orthonormal bases in $H$, we deduce that $\Delta^{-1}(\widetilde{\eta}_{j_0}\otimes \widetilde{\xi}_0) = \eta\otimes \xi$ must be an element in $\mathbb{T} {\eta}_{j_0}\otimes {\xi}_{j_0}$. That is, $\Delta^{-1}(\widetilde{\eta}_{j_0}\otimes \widetilde{\xi}_0) = \eta\otimes \xi = \mu {\eta}_{j_0}\otimes {\xi}_{j_0}$ for a unique $\mu \in \mathbb{T}$. Now, applying $(b)$ we get $$ \widetilde{\eta}_{j_0}\otimes \widetilde{\xi}_0 \in \{ \mu \ \widetilde{\eta}_{j_0}\otimes \widetilde{\xi}_{j_0}, \overline{\mu} \ \widetilde{\eta}_{j_0}\otimes \widetilde{\xi}_{j_0}\},$$ which is impossible. Similar arguments show that $\{\widetilde{\eta}_j: j\in J\}$ is an orthogonal basis in $H'$.
\end{proof}

\begin{corollary}\label{c minimal partial isometries in arbitrary dimension}
Let $H$ and $H'$ be complex Hilbert spaces, let $1< p <\infty$, $p\neq 2$, and
let $\Delta: S(C_p(H))\to S(C_p(H'))$ be a surjective isometry. Then the restriction $\Delta|_{_{\mathcal{U}_{min} (H)}}: \mathcal{U}_{min} (H)\to \mathcal{U}_{min} (H')$ is a surjective isometry.
\end{corollary}

\begin{remark}\label{remark on a couple of minimal partial isometries} Let us briefly recall some basic facts on the relative position of two minimal partial isometries. Let $v$ and $e$ be two minimal partial isometries in $C_p(H)\subseteq K(H)\subseteq B(H)$, where $\dim(H)\geq 2$. Arguing as in the proof of \cite[Lemma 3.3]{FerGarPeVill17} {\rm(}see also \cite[Proposition 3.3 and its proof]{FerPe17}{\rm)}, we can choose two orthonormal systems $\{\eta_1, \eta_2\}$ and $\{\xi_1, \xi_2\}$ in $H$ to write $e$ and $v$ in the form $e = \eta_1\otimes \xi_1$, $v = \widetilde{\eta}_1\otimes \widetilde{\xi}_1,$ and $$v=\alpha  v_{11}+ \beta v_{12}+ \delta v_{22}+\gamma v_{21},$$ where
$e_1= v_{11}$, $v_{12} = {\eta}_2 \otimes {\xi}_1$, $v_{21} = {\eta}_1 \otimes {\xi}_2$, $v_{22} = {\eta}_2 \otimes {\xi}_2$,
$\alpha= \langle \xi_1 | \widetilde{\xi}_1 \rangle \langle  \widetilde{\eta}_1 | \eta_1 \rangle,$ $ \beta = \langle \xi_1 | \widetilde{\xi}_1 \rangle \langle  \widetilde{\eta}_1 | \eta_2 \rangle,$  $\gamma= \langle \xi_2 | \widetilde{\xi}_1 \rangle \langle  \widetilde{\eta}_1 | \eta_1 \rangle,$ $\delta = \langle \xi_2 | \widetilde{\xi}_1 \rangle \langle  \widetilde{\eta}_1 | \eta_2 \rangle \in \mathbb{C},$ with $|\alpha|^2+ |\beta|^2+|\gamma|^2+ |\delta|^2$ $=|\langle \xi_1 | \widetilde{\xi}_1 \rangle|^2 \|\widetilde{\eta}_1\|^2+ |\langle \xi_2 | \widetilde{\xi}_1 \rangle|^2 \|\widetilde{\eta}_1 \|^2 = \| \widetilde{\xi}_1\|^2=1$, and $\alpha \delta= \beta \gamma$. The appropriate matrix representation in these two systems reads as follows: $e= \left(
                  \begin{array}{cc}
                    1 & 0 \\
                    0 & 0 \\
                  \end{array}
                \right),$ and $v=\left(
                  \begin{array}{cc}
                    \alpha & \gamma \\
                    \beta & \delta \\
                  \end{array}
                \right).$\smallskip

We can now enlarge the orthonormal systems $\{\eta_1, \eta_2\}$ and $\{\xi_1, \xi_2\}$ to get two orthonormal bases $\{\eta_1, \eta_2\}\cup\{ \eta_j : j\in J\}$ and $\{\xi_1, \xi_2\}\cup\{ \xi_j : j\in J\}$ in $H$. If we set $e_j :=\eta_j\otimes \xi_j$ $j\in \{1,2\}\cup J$, then $\displaystyle \bigcap_{j\in J} \{ e_j\}^{\perp} \cong C_p(H_1),$ where $H_1$ is a two dimensional complex Hilbert space, with $e,v \in S\left(\displaystyle \bigcap_{j\in J} \{ e_j\}^{\perp}\right) \cong S(C_p(H_1))\subset S(C_p(H))$.
\end{remark}

We continue with our study on the relative position of the image of an arbitrary minimal partial isometry $v$ and the images of the elements in $\mathbb{T} v$ under a surjective isometry between the spheres.

\begin{proposition}\label{p minimal partial isometries uniformly homogeneous} Let $H$ and $H'$ be complex Hilbert spaces, let $p \in(1,\infty)\backslash\{2\}$, and let $\Delta: S(C_p(H))\to S(C_p(H'))$ be a surjective isometry. Then one, and precisely one, of the following statements holds:\begin{enumerate}[$(a)$]\item $\Delta (\lambda v) = \lambda \Delta(v)$ for every $\lambda\in \mathbb{T}$ and every $v\in \mathcal{U}_{min} (H);$
\item $\Delta (\lambda v) = \overline{\lambda} \Delta(v)$ for every $\lambda\in \mathbb{T}$ and every $v\in \mathcal{U}_{min} (H).$
\end{enumerate}
\end{proposition}

\begin{proof} Let us define $\mathfrak{D}_1:= \{ v\in \mathcal{U}_{min} (H): \Delta (\lambda v) = \lambda \Delta(v) \hbox{ for all } \lambda\in \mathbb{T}\}$ and $\mathfrak{D}_2:= \{ v\in \mathcal{U}_{min} (H): \Delta (\lambda v) = \overline{\lambda} \Delta(v)  \hbox{ for all } \lambda\in \mathbb{T} \}$. Proposition \ref{p minimal partial isometries in arbitrary dimension}$(b)$ and $(c)$ implies that $\mathfrak{D}_1 \stackrel{\circ}{\cup} \mathfrak{D}_2$.\smallskip

We pick $v\in \mathfrak{D}_1$ and $w\in \mathcal{U}_{min} (H)$ with $\|v-w\|_p <\varepsilon$ for some $0<\varepsilon <1$.  Proposition \ref{p minimal partial isometries in arbitrary dimension}$(b)$ assures that $\Delta (i w) = i \Delta(w)$ or $\Delta (i w) = -i \Delta(w)$. In the second case we deduce from Remark \ref{r strictly convex commutes with -1} that $$\|v-w\|_p =\| \Delta (i v) -  \Delta(i w) \|_p = \| i \Delta (v) + i \Delta(w) \|_p  = \| \Delta (v) + \Delta(w) \|_p = \|v+w\|_p,$$ and hence $$1= \|v\|_p = \left\| \frac{v-w}{2} + \frac{v+w}{2} \right\|_p\leq \left\| \frac{v-w}{2}\right\|_p + \left\|\frac{v+w}{2} \right\|_p <\varepsilon<1,$$ which is impossible. Therefore $\Delta (i w) = i \Delta(w)$ and Proposition \ref{p minimal partial isometries in arbitrary dimension}$(c)$ proves that $w\in \mathfrak{D}_1$. We have therefore shown that $\mathfrak{D}_1$ is open. Similar arguments assure  that $\mathfrak{D}_2$ is also open.\smallskip

Finally, since it is well known that $\mathcal{U}_{min} (H)$ is a connected set, then $\mathcal{U}_{min} (H) = \mathfrak{D}_1$ or $\mathcal{U}_{min} (H) = \mathfrak{D}_2$, which concludes the proof.
\end{proof}

The conclusion of the above proposition in the case $p=1$ was established in \cite{FerGarPeVill17}.\smallskip

The next lemma is probably part of the folklore in the theory of inequalities, however we do not know an explicit reference. For each $1\leq j\leq m$ ($j,m\in \mathbb{N}$), let $e_{j}$ denote the element in $\ell_p^m$ whose $j$th-component is $1$ and all its other component are zero.

\begin{lemma}\label{l lp commutative} Let $\widehat{a}$ be a hermitian element in $S(\ell_p^{2n})$, with $p\in (1,\infty)\backslash\{2\}$, and let $\gamma$ be a real number with $\gamma\geq 1$. Let us assume that $$\widehat{a}= \left(\frac{1}{2^{\frac1p}} \lambda_1 , \ldots, \frac{1}{2^{\frac1p}} \lambda_{n},-\frac{1}{2^{\frac1p}} \lambda_{n},\ldots,-\frac{1}{2^{\frac1p}} \lambda_{1} \right),$$ where $\lambda_{1}\geq \lambda_{2}\geq \ldots \geq \lambda_{n}\geq 0$ and $\displaystyle \sum_{j=1}^n \lambda_j^p =1$. We set ${a}= \left(\lambda_1 , \ldots, \lambda_{n}\right)\in S(\ell_p^{n}),$
$$\mathcal{U}_2^s(\ell_p^{2n})=\left\{ s_{ij} :=\frac{1}{2^{\frac1p}}(e_i-e_j): 1\leq i\neq j\leq 2n \right\}\subset S(\ell_p^{2n}),$$ and $\hbox{Proj}_1^{\pm}(\ell_p^{n})=\left\{ \pm e_j : 1\leq j\leq n \right\}\subset S(\ell_p^{n})$. Then the following statements hold:
\begin{enumerate}[$(a)$] \item The minimum value of the mapping $k_a : \hbox{Proj}_1^{\pm}(\ell_p^{n})\to \mathbb{R}_0^+$, $k_a(z) :=\|a-\gamma z\|_p^p$ is $\displaystyle (\gamma-\lambda_1)^p+\sum_{j=2}^n \lambda_j^p = (\gamma-\lambda_1)^p + 1-\lambda_1^p$, and it is attained only at those points $z\in \hbox{Proj}_1^{\pm}(\ell_p^{n})$ satisfying $z= e_i$ with $\lambda_i = \lambda_1$;
\item The minimum value of the mapping $h_{\widehat{a}}: \mathcal{U}_2^s(\ell_p^{2n}) \to \mathbb{R}_0^+$, $h_{\widehat{a}}(s_{ij}) :=\|\widehat{a}-\gamma s_{ij}\|_p^p$ is $\displaystyle (\gamma-\lambda_1)^p+\sum_{j=2}^n \lambda_j^p = (\gamma-\lambda_1)^p + 1-\lambda_1^p$, and it is attained only at those point $s_{ij}\in \mathcal{U}_2^s$ satisfying $1\leq i\leq n$, $n+1\leq j\leq 2n$, $\lambda_i = \lambda_j = \lambda_1$.
\end{enumerate}
\end{lemma}

\begin{proof} $(a)$ Let us pick $e_j$ with $\lambda_j<\lambda_1$ and $e_i$ with $\lambda_i = \lambda_1$. It is easy to check that, since $(\gamma-\lambda_j)^p > (\gamma-\lambda_i)^p= (\gamma-\lambda_1)^p$, and the function $t\mapsto (\gamma -t)^p+ 1 -t^p$ ($t\in ]0, 1]$) is strictly decreasing, we have $$k_a(e_j) = \|a-\gamma e_j\|_p^p = \sum_{k\neq j} \lambda_k^p + (\gamma - \lambda_j )^p = 1- \lambda_j^p + (\gamma -\lambda_j)^p $$ $$> 1- \lambda_1^p + (\gamma -\lambda_1 )^p = k_a(e_1) = k_a(e_i).$$

On the other hand, for any $j$ we have $$k_a(-e_j) = \|a+\gamma e_j\|_p^p = \sum_{k\neq j} \lambda_k^p + (\lambda_j +\gamma )^p >  \sum_{k\neq j} \lambda_k^p + (\gamma - \lambda_j )^p= \|a-\gamma e_j\|_p^p= k_a(e_j) ,$$ which concludes the proof of $(a)$.\smallskip

$(b)$ Let us take $i,j\in \{1,\ldots,n \}$ and compute $$h_{\widehat{a}}(s_{ij}) :=\|\widehat{a}-\gamma s_{ij}\|_p^p =\sum_{k=1, k\neq i,j}^n  \frac{\lambda_{k}^p}{2} + \frac{ (\gamma -\lambda_i )^p}{2} +  \frac{(\lambda_j +\gamma )^p}{2} + \sum_{k=1}^n  \frac{\lambda_{k}^p}{2} $$
$$=\frac12 - \frac{\lambda_{i}^p}{2} - \frac{\lambda_{j}^p}{2} + \frac{ (\gamma - \lambda_i )^p}{2} +  \frac{(\lambda_j +\gamma )^p}{2} +\frac12 $$ $$\geq \frac12 \left(1 - {\lambda_{1}^p} +  (\gamma -\lambda_1 )^p \right) +\frac12 \left(1- {\lambda_{j}^p}  +  {(\lambda_j +\gamma )^p}\right) $$ $$ > \frac12 \left(1 - {\lambda_{1}^p} +  (\gamma -\lambda_1 )^p \right) +\frac12 \left(1- {\lambda_{j}^p}  +  {(\gamma - \lambda_j )^p}\right)\geq 1 - {\lambda_{1}^p} +  (\gamma - \lambda_1 )^p. $$

We can similarly prove that for $i,j\in \{n+1,\ldots, 2 n \}$ we have $$h_{\widehat{a}}(s_{ij}) > 1 - {\lambda_{1}^p} +  (\gamma -\lambda_1 )^p. $$

It is not hard to check that for $(i,j)\in \{1,\ldots,n \}\times\{n+1,\ldots,2n\}$ (respectively, for $(j,i)\in \{1,\ldots,n \}\times\{n+1,\ldots,2n\}$) the identity $$ h_{\widehat{a}}(s_{ij}) = \|\widehat{a}-\gamma s_{ij}\|_p^p = \frac12 \|{a}-\gamma e_i\|_p^p +\frac12 \|{a}-\gamma e_{j}\|_p^p =\frac12 k_a(e_i) + \frac12 k_a(e_j),$$ (respectively, $$ h_{\widehat{a}}(s_{ij}) = \|\widehat{a}-\gamma s_{ij}\|_p^p = \frac12 \|{a}+\gamma e_i\|_p^p +\frac12 \|{a}+\gamma e_{j}\|_p^p =\frac12 k_a(-e_i) + \frac12 k_a(-e_j))$$ holds. Finally, the desired statement is a straight consequence of the above identities, what is proved in the first two paragraphs, and statement $(a)$.
\end{proof}

We continue with another technical result.

\begin{proposition}\label{p min value} Let $H$ be a finite dimensional complex Hilbert space, and let $p$ be a real number in $(1,\infty)\backslash\{2\}$. Given $a\in C_p(H)$ and a real constant $\gamma \geq 1$, we consider the mapping $f_a : \mathcal{U}_{min} (H)\to \mathbb{R}_0^+$ defined by the assignment $e\mapsto f_a(e):=\|a-\gamma e\|_p^p$. Suppose $\displaystyle a = \sum_{j=1}^n \sigma_j(a) e_j$, where $e_1,\ldots,e_n$ are mutually orthogonal minimal partial isometries, $\sigma_1(a)= \ldots = \sigma_{j_0} (a) > \sigma_{j_0+1} (a)\geq \ldots \geq \sigma_n(a) \geq 0$ are the singular values of $a$ and $\displaystyle \|a\|_p^p = \sum_{j=1}^n \sigma_j(a)^p =1$. Then \begin{equation}\label{eq minimum value fa} \min_{e\in \mathcal{U}_{min} (H)}\!\! f_a (e) = f_a (e_j) = (\gamma-\sigma_1(a))^p+\sum_{j=2}^n \sigma_j(a)^p = (\gamma-\sigma_1(a))^p + 1-\sigma_1(a)^p,
\end{equation}  for every $1\leq j\leq j_0$. Furthermore, if $e_m$ denotes the partial isometry $\displaystyle e_m=\sum_{j=1}^{j_0} e_j$, then \begin{equation}\label{eq points of minimum value for fa} \hbox{$f_a$ attains its minimum value at $v\in \mathcal{U}_{min} (H)$}\hbox{ if, and only if, } v\leq e_m.
\end{equation}
\end{proposition}

\begin{proof} We observe that, since $H$ is finite dimensional, the set $\mathcal{U}_{min} (H)$ is $\|.\|_p$-compact. Obviously $f_a$ is $\|.\|_p$-continuous, and thus $f_a$ attains its maximum and minimum values in $\mathcal{U}_{min} (H)$. Next we shall determine the points in $\mathcal{U}_{min} (H)$ at which $f_a$ attains its minimum value.\smallskip

We have assumed that $\sigma_1(a)= \ldots = \sigma_{j_0}(a) > \sigma_{j_0+1}(a)\geq \ldots \geq \sigma_n(a)$ for some $j_0\in \{1,\ldots,n\}$. Having in mind that $\|.\|_p$ is unitarily-invariant, we deduce from \cite[Theorem 9.8]{Bhat2007}, applied to $a$ and $\gamma e$, that $$ \| \hbox{diag} (\sigma_1(a), \ldots, \sigma_n(a)) - \hbox{diag} (\sigma_1(\gamma e), \ldots, \sigma_n(\gamma e)) \|_p \leq \|a-\gamma e\|_p,$$ where $\hbox{diag} (., \ldots, .)$ stands for the diagonal matrix whose entries are given by the corresponding list. Since, clearly $\hbox{diag} (\sigma_1(\gamma e), \ldots, \sigma_n(\gamma e)) = \hbox{diag} (\gamma, 0, \ldots, 0) $, we get $$ f_a (e_j) = (\gamma-\sigma_1(a))^p + 1-\sigma_1(a)^p   $$ $$= \| \hbox{diag} (\sigma_1(a), \ldots, \sigma_n(a)) - \hbox{diag} (\sigma_1(\gamma e), \ldots, \sigma_n(\gamma e)) \|_p^p \leq \|a-\gamma e\|_p^p = f_a (e),$$ for every $1\leq j\leq j_0$, which proves \eqref{eq minimum value fa}.\smallskip

Let $e_m$ denote the partial isometry $\displaystyle e_m=\sum_{j=1}^{j_0} e_j$. Accordingly to our notation the support partial isometry of $a$, $\displaystyle  s(a) = \sum_{\sigma_j(a)\neq 0} e_j,$ satisfies $e_m\leq s(a)$. Let $v$ be any minimal partial isometry such that $v\leq e_m$, that is, $e_m = vv^*e_m v^*v + (1-vv^*) e_m (1-v^*v) = v + (1-vv^*) e_m (1-v^*v).$ Since $\displaystyle a= \sigma_1(a) v + \sigma_1(a) (e_m-v) +\sum_{j=j_0 +1}^n \sigma_j(a) e_j,$ we can easily compute that $$ f_a (v) = \|a-\gamma v\|_p^p  = (\gamma -\sigma_1(a))^p+\sum_{j=2}^n \sigma_j(a)^p = \min_{e\in \mathcal{U}_{min} (H)}\!\! f_a (e) = f_a (e_j).$$ We have therefore shown that $f_a$ attains its minimum value at every minimal partial isometry $v\in \mathcal{U}_{min} (H)$ with $v\leq e_m$.\smallskip

In order to prove \eqref{eq points of minimum value for fa} we shall make use of an ingenious device due to Wielandt (see for example \cite[page 24]{Bhat2007}). Let $\widetilde{a}$ denote the matrix in $M_2(M_n(\mathbb{C}))= M_{2n}(\mathbb{C})$ defined by $\widetilde{a} = \frac{1}{2^{\frac1p}} \left(
                                                                                                              \begin{array}{cc}
                                                                                                                0 & a \\
                                                                                                                a^* & 0 \\
                                                                                                              \end{array}
                                                                                                            \right)$.
It is known that $\widetilde{a}$ is hermitian and the eigenvalues of $\widetilde{a}$ are precisely the singular values of $\frac{1}{2^{\frac1p}} a$ together with their negatives (cf. \cite[page 24]{Bhat2007}). Consequently, $$\|\widetilde{a}\|_p^p = 2 \sum_{j=1}^n \left(\frac{1}{2^{\frac1p}}\right)^p \sigma_j(a)^p = \|a\|_p^p =1.$$

Let us consider the set $$\mathcal{U}_{2}^{sym} (M_{2n}(\mathbb{C})) = \Big\{ \frac{1}{2^{\frac1p}} p_1 - \frac{1}{2^{\frac1p}} p_2 : p_1,p_2\in \hbox{Proj}_1(M_{2n}(\mathbb{C})), \ p_1\perp p_2 \Big\}.$$ Clearly, $\mathcal{U}_{2}^{sym} (M_{2n}(\mathbb{C})) \subseteq S(C_p(M_{2n}(\mathbb{C}))).$ Let $g_{\widetilde{a}} : \mathcal{U}_{2}^{sym} (M_{2n}(\mathbb{C})) \to \mathbb{R}_0^+$ be the function defined by $g_{\widetilde{a}} (z) = \| {\widetilde{a}} -\gamma z\|_p^p$ ($\forall z\in \mathcal{U}_{2}^{sym} (M_{2n}(\mathbb{C})))$. A compactness argument shows that $g_{\widetilde{a}}$ attains its minimum value at $\mathcal{U}_{2}^{sym} (M_{2n}(\mathbb{C}))$. We shall prove next that \begin{equation}\label{eq minimum value qwidetildea} \min_{z\in \mathcal{U}_{2}^{sym} (M_{2n}(\mathbb{C}))}\!\! g_{\widetilde{a}} (z) =  (\gamma -\sigma_1(a))^p+\sum_{j=2}^n \sigma_j(a)^p
\end{equation} $$ = (\gamma -\sigma_1(a))^p + 1-\sigma_1(a)^p= \min_{e\in \mathcal{U}_{min} (H)}\!\! f_a (e).$$ Fix $z\in \mathcal{U}_{2}^{sym} (M_{2n}(\mathbb{C}))$. Since $\|.\|_p$ is unitarily-invariant, a new application of \cite[Theorem 9.8 or Theorem 9.7]{Bhat2007}, applied to $\widetilde{a}$ and $\gamma  z$, that $$ \| \hbox{diag} (\sigma_1(\widetilde{a}), \ldots, \sigma_{2n}(\widetilde{a})) - \hbox{diag} (\sigma_1(\gamma z), \ldots, \sigma_{2n}(\gamma  z)) \|_p \leq \|\widetilde{a}-\gamma  z\|_p,$$ where $$(\sigma_1(\widetilde{a}), \ldots, \sigma_{2n}(\widetilde{a})) = \left(\frac{1}{2^{\frac1p}} \sigma_1({a}), \frac{1}{2^{\frac1p}}\sigma_1({a}), \ldots, \frac{1}{2^{\frac1p}} \sigma_{n}({a}), \frac{1}{2^{\frac1p}} \sigma_{n}({a})\right),$$ and $$(\sigma_1(\gamma  z), \ldots, \sigma_{2n}(\gamma  z)) = \left(\frac{\gamma}{2^{\frac1p}},\frac{\gamma}{2^{\frac1p}},0, \ldots, 0\right).$$ We therefore have $$ (\gamma -\sigma_1(a))^p + 1-\sigma_1(a)^p= (\gamma -\sigma_1(a))^p+\sum_{j=2}^n \sigma_j(a)^p $$ $$= 2 \left(\frac{1}{2^{\frac1p}}\right)^p (\gamma -\sigma_1(a))^p+ 2 \sum_{j=2}^n \left(\frac{1}{2^{\frac1p}}\right)^p \sigma_j(a)^p $$ $$= \| \hbox{diag} (\sigma_1(\widetilde{a}), \ldots, \sigma_{2n}(\widetilde{a})) - \hbox{diag} (\sigma_1(\gamma z), \ldots, \sigma_{2n}(\gamma z)) \|_p^p \leq \|\widetilde{a}-\gamma z\|_p^p = g_{\widetilde{a}} (z),$$ which proves \eqref{eq minimum value qwidetildea}.\smallskip

Let us also observe that, given a minimal partial isometry $v\in \mathcal{U}_{min} (H)$ satisfying $f_a(v) = \min_{e\in \mathcal{U}_{min} (H)} f_a (e)$, the matrix $\widetilde{v} = \frac{1}{2^{\frac1p}} \left(
                                                                                                              \begin{array}{cc}
                                                                                                                0 & v \\
                                                                                                                v^* & 0 \\
                                                                                                              \end{array}
                                                                                                            \right)$
lies in $\mathcal{U}_{2}^{sym} (M_{2n}(\mathbb{C}))$ and, by orthogonality, the Clarkson-McCarthy inequalities (compare \eqref{Clarkson-McCarthy ineq 0pleq2} and \eqref{Clarkson-McCarthy ineq 2p infty}), and the invariance of the $p$-norm under taking the involution $*$ (see \cite[Theorem 1.3]{McCarthy67}), we get $$g_{\widetilde{a}} (\widetilde{v}) = \| \widetilde{a} -\gamma \widetilde{v}\|_p^p = \left(\frac{1}{2^{\frac1p}}\right)^p \|a-\gamma  v\|_p^p + \left(\frac{1}{2^{\frac1p}}\right)^p \|a^*- \gamma  v^*\|_p^p  $$ $$= \|a- \gamma v\|_p^p = \min_{e\in \mathcal{U}_{min} (H)} f_a (e) = \min_{z\in \mathcal{U}_{2}^{sym} (M_{2n}(\mathbb{C}))}\!\! g_{\widetilde{a}} (z).$$

Now, let us fix two orthogonal minimal projections $q_1,q_2\in \hbox{Proj}_1(M_{2n}(\mathbb{C}))$ and set $b=  \frac{1}{2^{\frac1p}} q_1 - \frac{1}{2^{\frac1p}} q_2 \in S(C_p(M_{2n}(\mathbb{C})))$. It is clear that, accordingly to the terminology in \cite{BhatSemr96}, the \emph{unitary orbit} of $b$ in $S(C_p(M_{2n}(\mathbb{C})))$ is the set $$\left\{ \widetilde{u} b \widetilde{u}^* : \widetilde{u} \hbox{ unitary in } M_{2n}(\mathbb{C}) \right\}$$ and coincides with our set $\mathcal{U}_{2}^{sym} (M_{2n}(\mathbb{C}))$. R. Bhatia and P. \v{S}emrl prove in \cite[Theorem 1]{BhatSemr96} that if an element $b_0\in \mathcal{U}_{2}^{sym} (M_{2n}(\mathbb{C}))$ is a critical point for the mapping $g_{\widetilde{a}},$ then $b_0$ commutes with $\widetilde{a}$. By applying this conclusion to $b_0=\widetilde{v},$ we deduce that $\widetilde{v}$ and $\widetilde{a}$ commute, and consequently, $a^* v = v^* a$ and $ a v^* = v a^*$.\smallskip

Furthermore, since $\widetilde{v}$ and $\widetilde{a}$ commute, $\widetilde{a}$ is hermitian, and $\widetilde{v}$ is a positive multiple of a rank-2 hermitian partial isometry, we can easily deduce the existence of two orthogonal minimal projections $r_1$ and $r_2$ in $M_{2n} (\mathbb{C})$ satisfying $\widetilde{v} = \frac{1}{2^{\frac1p}} (r_1-r_2)$, $r_1$ and $r_2$ commute with $\widetilde{a},$ and consequently, $\widetilde{a} = r_1 \widetilde{a} r_1 + r_2 \widetilde{a} r_2 + (1-r_1-r_2) \widetilde{a} (1-r_1-r_2)$. By applying a joint spectral resolution of $\widetilde{a}$ and $\widetilde{v}$, and keeping the notation in Lemma \ref{l lp commutative}, we can represent $\widetilde{a}$ and $\widetilde{v}$ in a commutative $\ell_p^{2n}$ space, with a representation satisfying the following properties:
\begin{enumerate}[$(1)$]
\item $\widetilde{a}= \left(\frac{1}{2^{\frac1p}} \sigma_1(a) , \ldots, \frac{1}{2^{\frac1p}} \sigma_{n}(a),-\frac{1}{2^{\frac1p}} \sigma_{n}(a),\ldots,-\frac{1}{2^{\frac1p}} \sigma_{1}(a) \right)$;
\item $\widetilde{v}\in \mathcal{U}_2^s(\ell_p^{2n});$
\item $\displaystyle\min_{z\in  \mathcal{U}_2^s(\ell_p^{2n})} h_{\widetilde{a}}(z) $ $\displaystyle=(\gamma-\sigma_1(a))^p+\sum_{j=2}^n \sigma_j^p(a) $ $\displaystyle= (\gamma-\sigma_1(a))^p + 1-\sigma_1^p(a) $ $=\|\widetilde{a}-\gamma \widetilde{v}\|_p^p$ $\displaystyle=\min_{z\in \mathcal{U}_{2}^{sym} (M_{2n}(\mathbb{C}))} g_{\widetilde{a}} (z)$.
\end{enumerate}

We are in position to apply Lemma \ref{l lp commutative}. We therefore conclude that $\widetilde{v}= \frac{1}{2^{\frac1p}}(e_i-e_j)$ with $1\leq i\leq n$, $n+1\leq j\leq 2n$, $\sigma_i(a) = \sigma_j(a) = \sigma_1(a)$, which implies that $v \leq e_m$. This concludes the proof of \eqref{eq points of minimum value for fa}.
\end{proof}

We shall need later an appropriate generalization of \cite[Lemma in page 3]{Nag2013}. Our next result extends the just quoted result in \cite{Nag2013} to general matrices in the unit sphere of $C_p(M_n(\mathbb{C}))$.

\begin{proposition}\label{p a la Nagy} Let $H$ be a finite dimensional complex Hilbert space, and let $p$ be a real number in $(1, \infty)\backslash\{2\}$. We fix a positive $\gamma \geq 1$. If $a,b\in S(C_p(H))$ satisfy $\|a-\gamma e\|_p = \|b-\gamma e\|_p$ for every $e$ in $\mathcal{U}_{min} (H)$, then $a = b$.
\end{proposition}

\begin{proof} Let $a,b\in S(C_p(H))$ satisfying $\|a-\gamma  e\|_p = \|b-\gamma  e\|_p$ for every $e\in\mathcal{U}_{min} (H)$. Our aim is to show that $a= b$. We shall argue by induction on the dimension of $H$. The case dim$(H)=1$ is clear. Let us write $\displaystyle a = \sum_{j=1}^n \sigma_j(a) e_j$, and $\displaystyle  b = \sum_{j=1}^n \sigma_j(b) v_j,$ where $e_1,\ldots,e_n$ and $v_1,\ldots,v_n$ are two families of mutually orthogonal minimal partial isometries, $\sigma_1(a)\geq \ldots \geq \sigma_n(a) \geq 0$ and $\sigma_1(b)\geq \ldots \geq \sigma_n(b) \geq 0$ are the singular values of $a$ and $b$, respectively. As before $\displaystyle e_m = \sum_{j, \sigma_1(a) = \sigma_j(a)} e_j$ and $\displaystyle  v_m = \sum_{j, \sigma_1(b) = \sigma_j(b)} v_j.$ It follows from the assumptions that the mappings $f_a, f_b : \mathcal{U}_{min} (H)\to \mathbb{R}_0^+,$ $ f_a(e)=\|a-\gamma e\|_p^p=\|b-\gamma e\|_p^p=f_b(e)$ coincide. Thus by Proposition \ref{p min value}\eqref{eq minimum value fa} we have $$\min_{e\in \mathcal{U}_{min} (H)}\!\! f_a (e)  = (\gamma -\sigma_1(a))^p + 1-\sigma_1(a)^p= \min_{e\in \mathcal{U}_{min} (H)}\!\! f_b (e)  = (\gamma -\sigma_1(b))^p + 1-\sigma_1(b)^p,$$ and by Proposition \ref{p min value}\eqref{eq points of minimum value for fa} $$\hbox{$f_a$ attains its minimum value at $v\in \mathcal{U}_{min} (H)$}\Leftrightarrow v\leq e_m,$$ and $$\hbox{$f_b$ attains its minimum value at $v\in \mathcal{U}_{min} (H)$}\Leftrightarrow v\leq v_m.$$ \smallskip

The equality $f_a=f_b$ now implies that $e_m= v_m$.\smallskip

As observed by G. Nagy in the proof of \cite[Lemma in page 3]{Nag2013}, the values of $\sigma_1(a)$ and $\sigma_1(b)$ can be recovered from the mappings $f_a$ and $f_b$. Namely, the function $t\mapsto (\gamma -t)^p+ 1 -t^p$ ($t\in ]0, 1]$) is strictly decreasing and thus injective. Therefore, the equality $$ (\gamma -\sigma_1(a))^p + 1-\sigma_1(a)^p=  (\gamma -\sigma_1(b))^p + 1-\sigma_1(b)^p,$$ holds if and only if $\sigma_1(a)=\sigma_1(b)$. Therefore, $e_m e_m^*a e_m^* e_m = v_m v_m^*b v_m^* v_m$. If $\sigma_1(a)=\sigma_1(b)=1 $ or $e_m e_m^*=1$ we get $a=b$. We may assume that $\sigma_1(a)=\sigma_1(b)\neq 1$ and $e_m e_m^*\neq 1$, and hence $$0< \|(1-e_m e_m^*)b (1-e_m^* e_m)\|_p = \|(1-e_m e_m^*)a (1-e_m^* e_m)\|_p<1.$$

Now, given a minimal partial isometry $e\in \mathcal{U}_{min} (H)$ with $e\perp e_m=v_m$, we deduce from the orthogonality of $e_m e_m^*a e_m^* e_m = v_m v_m^*b v_m^* v_m$ and $e$ that $$\|e_m e_m^*a e_m^* e_m\|_p^p + \|(1-e_m e_m^*)a (1-e_m^* e_m) - \gamma e\|_p^p =\|a-\gamma e\|_p^p  $$ $$= \|b-\gamma e\|_p^p= \|e_m e_m^*b e_m^* e_m\|_p^p + \|(1-e_m e_m^*)b (1-e_m^* e_m) -\gamma e\|_p^p,$$ and consequently $$\left\|\frac{(1-e_m e_m^*)a (1-e_m^* e_m)}{\|(1-e_m e_m^*)a (1-e_m^* e_m)\|_p} - \frac{\gamma}{\|(1-e_m e_m^*)a (1-e_m^* e_m)\|_p} e\right\|_p $$ $$= \left\|\frac{(1-e_m e_m^*)b (1-e_m^* e_m)}{\|(1-e_m e_m^*)b (1-e_m^* e_m)\|_p} -\frac{\gamma}{\|(1-e_m e_m^*)b (1-e_m^* e_m)\|_p} e\right\|_p^p,$$ for every minimal partial isometry $e\in \mathcal{U}_{min} (H)$ with $e\perp e_m=v_m$. The induction hypothesis implies that $(1-e_m e_m^*)b (1-e_m^* e_m) = (1-e_m e_m^*) a (1-e_m^* e_m)$, which proves $a=b$ and finishes the proof.
\end{proof}

The case in which $H$ is a 2-dimensional complex Hilbert space is treated independently.

\begin{theorem}\label{t Tingley's for Cp with H 2dimensional} Let $H$ be a 2-dimensional complex Hilbert space, let
$p\in (1, \infty)\backslash\{2\}$ be a fixed real number, and let $\Delta: S(C_p(H))\to S(C_p(H))$ be a surjective isometry.
Then there exists a complex linear or a conjugate linear surjective isometry $T:C_p(H)\to C_p(H)$ whose restriction
to $S(C_p(H))$ is $\Delta$.
\end{theorem}

\begin{proof} We can identify $C_p(H)$ with $M_2(\mathbb{C})$ equipped with the norm $\|.\|_p$. The elements $\Delta \left(
                                                                          \begin{array}{cc}
                                                                            1 & 0 \\
                                                                            0 & 0 \\
                                                                          \end{array}
                                                                        \right)$ and $\Delta \left(
                                                                          \begin{array}{cc}
                                                                            0 & 0 \\
                                                                            0 & 1 \\
                                                                          \end{array}
                                                                        \right)$ are two orthogonal minimal partial isometries. By composing $\Delta$ with a surjective linear isometry on $M_2(\mathbb{C})$ we may assume that $\Delta \left(
                                                                          \begin{array}{cc}
                                                                            1 & 0 \\
                                                                            0 & 0 \\
                                                                          \end{array}
                                                                        \right) = \left(
                                                                          \begin{array}{cc}
                                                                            1 & 0 \\
                                                                            0 & 0 \\
                                                                          \end{array}
                                                                        \right)$ and $\Delta \left(
                                                                          \begin{array}{cc}
                                                                            0 & 0 \\
                                                                            0 & 1 \\
                                                                          \end{array}
                                                                        \right) = \left(
                                                                          \begin{array}{cc}
                                                                            0 & 0 \\
                                                                            0 & 1 \\
                                                                          \end{array}
                                                                        \right).$ We may also assume via Proposition \ref{p minimal partial isometries uniformly homogeneous} that $\Delta (\lambda e ) = \lambda \Delta(e)$ for every $e\in \mathcal{U}_{min} (H)$, $\lambda\in\mathbb{T}$.\smallskip

We set $\displaystyle a = \frac{1}{2^{\frac1p}} \left(
                                                                          \begin{array}{cc}
                                                                            1 & 0 \\
                                                                            0 & 1 \\
                                                                          \end{array}
                                                                        \right) =\sum_{j=1}^2 \sigma_j(a) e_j$ and $\displaystyle b = \Delta(a)=\sum_{j=1}^2 \sigma_j(b) v_j,$ where $\{e_1,e_2\}$ and $\{v_1,v_2\}$ are two families of mutually orthogonal minimal partial isometries, while $\sigma_j(a)$ and $\sigma_j(b)$ stand for the singular values of $a$ and $b$, respectively.
Keeping the notation employed in Proposition \ref{p min value} we denote $\displaystyle e_m = \sum_{j, \sigma_1(a) = \sigma_j(a) } e_j\in B(H)$ and $\displaystyle v_m = \sum_{j, \sigma_1(b) = \sigma_j(b) } v_j\in B(H)$. Clearly $e_m = 1$ because $\sigma_1(a) =\sigma_2(a)$.\smallskip

Let us consider the mappings $f_a, f_b : \mathcal{U}_{min} (H)\to \mathbb{R}_0^+$ defined by $f_a(e):=\|a- e\|_p^p$ and $f_b(e):=\|b- e\|_p^p$. Proposition \ref{p minimal partial isometries in arbitrary dimension}$(a)$ assures that $\Delta (\mathcal{U}_{min} (H))= \mathcal{U}_{min} (H).$ Since $f_a (e)=\| a- e\|_p^p =\|\Delta(a)-\Delta(e)\|_p^p=\|b-\Delta(e)\|_p^p = f_b (\Delta(e)),$ we deduce that $$\left(1-\frac{1}{2^{\frac1p}}\right)^p +\frac12 =\min_{e\in \mathcal{U}_{min} (H)}\!\! f_a (e) = \min_{e\in \mathcal{U}_{min} (H)}\!\! f_b (e) = (1-\sigma_1(b))^p +\sigma_2(b)^p.$$ Then it follows that $\sigma_1(b) = \sigma_2(b) = \frac{1}{2^{\frac1p}}$, and consequently $v_m = 1 = e_m$ and $a=b=\Delta(a)$. Furthermore, Proposition \ref{p minimal partial isometries in arbitrary dimension}$(b)$ gives
$$\left\{ v\in \mathcal{U}_{min} (H) : f_a(v) = \min_{e\in \mathcal{U}_{min} (H)}\!\! f_a (e) \right\} = \left\{ v\in \mathcal{U}_{min} (H) : v\leq 1 \right\}$$ $$=  \left\{ \hbox{minimal projections in } B(H)  \right\}= \left\{ v\in \mathcal{U}_{min} (H) : f_b(v) = \min_{e\in \mathcal{U}_{min} (H)}\!\! f_b (e) \right\},$$ which guarantees that $\Delta$ maps minimal projections in $B(H)$ to minimal projections in $B(H)$. \smallskip

Since $\Delta \left(
                \begin{array}{cc}
                  \frac12 & \frac12 \\
                  \frac12 & \frac12 \\
                \end{array}
              \right)$ must be a minimal projection, there exist $t\in (0,1)$ and $c\in \mathbb{T}$ satisfying $\Delta \left(
                \begin{array}{cc}
                  \frac12 & \frac12 \\
                  \frac12 & \frac12 \\
                \end{array}
              \right) = \left(
                          \begin{array}{cc}
                            t & c \sqrt{t(1-t)} \\
                            \overline{c} \sqrt{t(1-t)} & 1-t \\
                          \end{array}
                        \right).$ The hypothesis implies that $$2 (1-t)^{\frac{p}{2}}=\left\| \left(
                          \begin{array}{cc}
                            t-1 & c \sqrt{t(1-t)} \\
                            \overline{c} \sqrt{t(1-t)} & 1-t \\
                          \end{array}
                        \right)
               \right\|_p^p =$$
$$\left\| \left(
                          \begin{array}{cc}
                            t & c \sqrt{t(1-t)} \\
                            \overline{c} \sqrt{t(1-t)} & 1-t \\
                          \end{array}
                        \right) - \left(
                          \begin{array}{cc}
                            1 & 0 \\
                            0 & 0 \\
                          \end{array}
                        \right)
               \right\|_p^p =\left\| \Delta \left(
                \begin{array}{cc}
                  \frac12 & \frac12 \\
                  \frac12 & \frac12 \\
                \end{array}
              \right) - \Delta \left(
                          \begin{array}{cc}
                            1 & 0 \\
                            0 & 0 \\
                          \end{array}
                        \right)
               \right\|_p^p $$ $$=  \left\| \left(
                \begin{array}{cc}
                  \frac12 & \frac12 \\
                  \frac12 & \frac12 \\
                \end{array}
              \right) -  \left(
                          \begin{array}{cc}
                            1 & 0 \\
                            0 & 0 \\
                          \end{array}
                        \right)
               \right\|_p^p = \frac{2}{2^{\frac{p}{2}}},$$ and consequently $t=\frac12$ and $\Delta \left(
                \begin{array}{cc}
                  \frac12 & \frac12 \\
                  \frac12 & \frac12 \\
                \end{array}
              \right) = \frac12 \left(
                          \begin{array}{cc}
                            1 & c \\
                            \overline{c}  & 1 \\
                          \end{array}
                        \right).$ \smallskip

We consider the surjective linear isometry $T_0 : B(H) \to B(H)$ defined by $T_0 (x) := \left(
                                                                                          \begin{array}{cc}
                                                                                            1 & 0 \\
                                                                                            0 & c \\
                                                                                          \end{array}
                                                                                        \right) x \left(
                                                                                                    \begin{array}{cc}
                                                                                                      1 & 0 \\
                                                                                                      0 & \overline{c} \\
                                                                                                    \end{array}
                                                                                                  \right)$ and the surjective isometry $\Delta_1 = T_0 \Delta$.
The mapping $\Delta_1$ satisfies $\Delta_1 \left(\begin{array}{cc}
                                                                            1 & 0 \\
                                                                            0 & 0 \\
                                                                          \end{array}
                                                                        \right) = \left(
                                                                          \begin{array}{cc}
                                                                            1 & 0 \\
                                                                            0 & 0 \\
                                                                          \end{array}
                                                                        \right),$ $\Delta_1 \left(
                                                                          \begin{array}{cc}
                                                                            0 & 0 \\
                                                                            0 & 1 \\
                                                                          \end{array}
                                                                        \right) = \left(
                                                                          \begin{array}{cc}
                                                                            0 & 0 \\
                                                                            0 & 1 \\
                                                                          \end{array}
                                                                        \right),$ $\Delta_1 \left(
                \begin{array}{cc}
                  \frac12 & \frac12 \\
                  \frac12 & \frac12 \\
                \end{array}
              \right) = \left(
                \begin{array}{cc}
                  \frac12 & \frac12 \\
                  \frac12 & \frac12 \\
                \end{array}
              \right),$ $\Delta_1 (\lambda e ) = \lambda \Delta_1(e)$ for every $e\in \mathcal{U}_{min} (H)$, $\lambda\in \mathbb{T}$, and $\Delta_1$ maps minimal projections to minimal projections.\smallskip

Now, $\Delta_1 \left(
                \begin{array}{cc}
                  \frac12 & \frac{i}{2} \\
                  -\frac{i}{2} & \frac12 \\
                \end{array}
              \right)$ must be a minimal projection, and thus there exist $s\in (0,1)$ and $c\in \mathbb{T}$ satisfying $\Delta \left(
                \begin{array}{cc}
                  \frac12 & \frac{i}{2} \\
                  -\frac{i}{2} & \frac12 \\
                \end{array}
              \right) = \left(
                          \begin{array}{cc}
                            s & c \sqrt{s(1-s)} \\
                            \overline{c} \sqrt{s(1-s)} & 1-s \\
                          \end{array}
                        \right).$ The identities
                        $$\frac{2}{2^{\frac{p}{2}}}=\left\| \left(
                \begin{array}{cc}
                  0 & \frac{i-1}{2} \\
                  -\frac{i+1}{2} & 0 \\
                \end{array}
              \right)
               \right\|_p^p =\left\| \left(
                \begin{array}{cc}
                  \frac12 & \frac{i}{2} \\
                  -\frac{i}{2} & \frac12 \\
                \end{array}
              \right) -  \left(
                          \begin{array}{cc}
                            \frac12 & \frac12 \\
                            \frac12 & \frac12 \\
                          \end{array}
                        \right)
               \right\|_p^p $$
$$ =\left\| \Delta_1 \left(
                \begin{array}{cc}
                  \frac12 & \frac{i}{2} \\
                  -\frac{i}{2} & \frac12 \\
                \end{array}
              \right) - \Delta_1 \left(
                          \begin{array}{cc}
                            \frac12 & \frac12 \\
                            \frac12 & \frac12 \\
                          \end{array}
                        \right)
               \right\|_p^p $$ $$=  \left\| \left(
                          \begin{array}{cc}
                            s & c \sqrt{s(1-s)} \\
                            \overline{c} \sqrt{s(1-s)} & 1-s \\
                          \end{array}
                        \right) -  \left(
                          \begin{array}{cc}
                            \frac12 & \frac12 \\
                            \frac12 & \frac12 \\
                          \end{array}
                        \right)
               \right\|_p^p$$ $$=  \left\| \left(
                          \begin{array}{cc}
                            s-\frac12 & c \sqrt{s(1-s)}-\frac12 \\
                            \overline{c} \sqrt{s(1-s)}-\frac12 & \frac12-s \\
                          \end{array}
                        \right)
               \right\|_p^p $$ $$= 2 \left( \left(s-\frac12\right)^2 + \left| c \sqrt{s(1-s)} -\frac12 \right|^2  \right)^{\frac{p}{2}},$$ and
$$ \frac{2}{2^{\frac{p}{2}}}= \left\| \left(
                \begin{array}{cc}
                  \frac12 & \frac{i}{2} \\
                  -\frac{i}{2} & \frac12 \\
                \end{array}
              \right) - \left(
                          \begin{array}{cc}
                            1 & 0 \\
                            0 & 0 \\
                          \end{array}
                        \right)
               \right\|_p^p =\left\| \Delta_1 \left(
                \begin{array}{cc}
                  \frac12 & \frac{i}{2} \\
                  -\frac{i}{2} & \frac12 \\
                \end{array}
              \right) - \Delta_1 \left(
                          \begin{array}{cc}
                            1 & 0 \\
                            0 & 0 \\
                          \end{array}
                        \right)
               \right\|_p^p  $$ $$= \left\| \left(
                          \begin{array}{cc}
                            s & c \sqrt{s(1-s)} \\
                            \overline{c} \sqrt{s(1-s)} & 1-s \\
                          \end{array}
                        \right) - \left(
                          \begin{array}{cc}
                            1 & 0 \\
                            0 & 0 \\
                          \end{array}
                        \right)\right\|_p^p = 2 (1-s)^{\frac{p}{2}},$$ guarantee that $s=\frac12$ and $c=\pm i$. We have therefore shown that $$\Delta_1 \left(
                \begin{array}{cc}
                  \frac12 & \frac{i}{2} \\
                  -\frac{i}{2} & \frac12 \\
                \end{array}
              \right)\in \left\{ \left(
                \begin{array}{cc}
                  \frac12 & \frac{i}{2} \\
                  -\frac{i}{2} & \frac12 \\
                \end{array}
              \right), \left(
                \begin{array}{cc}
                  \frac12 & - \frac{i}{2} \\
                  \frac{i}{2} & \frac12 \\
                \end{array}
              \right) \right\}.$$ By composing with the transpose, if necessary, we may assume that $\Delta_1 \left(
                \begin{array}{cc}
                  \frac12 & \frac{i}{2} \\
                  -\frac{i}{2} & \frac12 \\
                \end{array}
              \right) =  \left(
                \begin{array}{cc}
                  \frac12 & \frac{i}{2} \\
                  -\frac{i}{2} & \frac12 \\
                \end{array}
              \right).$\smallskip

We consider next an arbitrary minimal projection $\left(
                          \begin{array}{cc}
                            t & c \sqrt{t(1-t)} \\
                            \overline{c} \sqrt{t(1-t)} & 1-t \\
                          \end{array}
                        \right)$ with $t\in (0,1)$, $c\in \mathbb{T}$. It follows from the above properties that $$\Delta_1 \left(
                          \begin{array}{cc}
                            t & c \sqrt{t(1-t)} \\
                            \overline{c} \sqrt{t(1-t)} & 1-t \\
                          \end{array}
                        \right)  = \left(
                          \begin{array}{cc}
                            s & d \sqrt{s(1-s)} \\
                            \overline{d} \sqrt{s(1-s)} & 1-s \\
                          \end{array}
                        \right),$$ for a unique pair $(s,d)\in (0,1)\times \mathbb{T}$. The equalities $$2 (1-t)^{\frac{p}{2}}= \left\| \left(
                          \begin{array}{cc}
                            t & c \sqrt{t(1-t)} \\
                            \overline{c} \sqrt{t(1-t)} & 1-t \\
                          \end{array}
                        \right) -  \left(
                \begin{array}{cc}
                  1 & 0 \\
                  0 & 0\\
                \end{array}
              \right)  \right\|_p^p$$
$$=\left\| \Delta_1 \left(
                          \begin{array}{cc}
                            t & c \sqrt{t(1-t)} \\
                            \overline{c} \sqrt{t(1-t)} & 1-t \\
                          \end{array}
                        \right) - \Delta_1 \left(
                \begin{array}{cc}
                  1 & 0 \\
                  0 & 0\\
                \end{array}
              \right)  \right\|_p^p $$ $$=\left\| \left(
                          \begin{array}{cc}
                            s & d \sqrt{s(1-s)} \\
                            \overline{d} \sqrt{s(1-s)} & 1-s \\
                          \end{array}
                        \right) -  \left(
                \begin{array}{cc}
                  1 & 0 \\
                  0 & 0\\
                \end{array}
              \right)  \right\|_p^p = 2 (1-s)^{\frac{p}{2}},$$

                        $$2 \left( \left(t-\frac12\right)^2 + \left| c \sqrt{t(1-t)}-\frac{i}{2} \right|^2 \right)^{\frac{p}{2}}$$ $$=\left\| \left(
                          \begin{array}{cc}
                            t & c \sqrt{t(1-t)} \\
                            \overline{c} \sqrt{t(1-t)} & 1-t \\
                          \end{array}
                        \right) -  \left(
                \begin{array}{cc}
                  \frac12 & \frac{i}{2} \\
                  -\frac{i}{2} & \frac12 \\
                \end{array}
              \right)  \right\|_p^p$$ $$=\left\| \Delta_1 \left(
                          \begin{array}{cc}
                            t & c \sqrt{t(1-t)} \\
                            \overline{c} \sqrt{t(1-t)} & 1-t \\
                          \end{array}
                        \right) - \Delta_1 \left(
                \begin{array}{cc}
                  \frac12 & \frac{i}{2} \\
                  -\frac{i}{2} & \frac12 \\
                \end{array}
              \right)  \right\|_p^p $$ $$=\left\|  \left(
                          \begin{array}{cc}
                            s & d \sqrt{s(1-s)} \\
                            \overline{d} \sqrt{s(1-s)} & 1-s \\
                          \end{array}
                        \right) - \left(
                \begin{array}{cc}
                  \frac12 & \frac{i}{2} \\
                  -\frac{i}{2} & \frac12 \\
                \end{array}
              \right)  \right\|_p^p $$ $$2 \left( \left(s-\frac12\right)^2 + \left| d \sqrt{s(1-s)}-\frac{i}{2} \right|^2 \right)^{\frac{p}{2}},$$ and

                        $$2 \left( \left(t-\frac12\right)^2 + \left| c \sqrt{t(1-t)}-\frac{1}{2} \right|^2 \right)^{\frac{p}{2}}$$ $$=\left\| \left(
                          \begin{array}{cc}
                            t & c \sqrt{t(1-t)} \\
                            \overline{c} \sqrt{t(1-t)} & 1-t \\
                          \end{array}
                        \right) -  \left(
                \begin{array}{cc}
                  \frac12 & \frac{1}{2} \\
                  \frac{1}{2} & \frac12 \\
                \end{array}
              \right)  \right\|_p^p$$ $$=\left\| \Delta_1 \left(
                          \begin{array}{cc}
                            t & c \sqrt{t(1-t)} \\
                            \overline{c} \sqrt{t(1-t)} & 1-t \\
                          \end{array}
                        \right) - \Delta_1 \left(
                \begin{array}{cc}
                  \frac12 & \frac{1}{2} \\
                  \frac{1}{2} & \frac12 \\
                \end{array}
              \right)  \right\|_p^p $$ $$=\left\|  \left(
                          \begin{array}{cc}
                            s & d \sqrt{s(1-s)} \\
                            \overline{d} \sqrt{s(1-s)} & 1-s \\
                          \end{array}
                        \right) - \left(
                \begin{array}{cc}
                  \frac12 & \frac{1}{2} \\
                  \frac{1}{2} & \frac12 \\
                \end{array}
              \right)  \right\|_p^p $$ $$=2 \left( \left(s-\frac12\right)^2 + \left| d \sqrt{s(1-s)}-\frac{1}{2} \right|^2 \right)^{\frac{p}{2}},$$ imply that $s=t$ and $c=d$. This shows that $\Delta_1 \left(
                          \begin{array}{cc}
                            t & c \sqrt{t(1-t)} \\
                            \overline{c} \sqrt{t(1-t)} & 1-t \\
                          \end{array}
                        \right) = \left(
                          \begin{array}{cc}
                            t & c \sqrt{t(1-t)} \\
                            \overline{c} \sqrt{t(1-t)} & 1-t \\
                          \end{array}
                        \right)$ for every $t$ and $c$ as above, that is, $\Delta_1$ in the identity mapping on rank one projections. \smallskip

For every unitary $u$ in $B(H)=M_2(\mathbb{C})$, there exist mutually orthogonal minimal projections $q_1,q_2 \in \hbox{Proj}_1(B(H))$ and $s,t\in \mathbb{T}$ such that $u = s q_1 + t q_2$. We set $\displaystyle \widetilde{a} = \frac{1}{2^{\frac1p}} u =\sum_{j=1}^2 \sigma_j(\widetilde{a}) e_j$ and $\displaystyle \widetilde{b} = \Delta_1(\widetilde{a})=\sum_{j=1}^2 \sigma_j(\widetilde{b}) v_j,$ where $\{e_1=s q_1,e_2 = t q_2\}$ and $\{v_1,v_2\}$ are two families of mutually orthogonal minimal partial isometries, while $\sigma_j(\widetilde{a})$ and $\sigma_j(\widetilde{b})$ stand for the singular values of $\widetilde{a}$ and $\widetilde{b}$, respectively.
As before, we denote $\displaystyle e_m = \sum_{j, \sigma_1(\widetilde{a}) = \sigma_j(\widetilde{a}) } e_j=u$ and $\displaystyle v_m = \sum_{j, \sigma_1(\widetilde{b}) = \sigma_j(\widetilde{b}) } v_j\in B(H)$.\smallskip

By repeating the arguments in the third paragraph of this proof to the mappings $f_{\widetilde{a}}, f_{\widetilde{b}} : \mathcal{U}_{min} (H)\to \mathbb{R}_0^+,$ $f_{\widetilde{a}}(e):=\|\widetilde{a}- e\|_p^p$ and $f_{\widetilde{b}}(e):=\|b- e\|_p^p$, we deduce that $\sigma_1(\widetilde{b}) = \sigma_2(\widetilde{b}) = \frac{1}{2^{\frac1p}}$, and consequently $v_m = u = e_m$ and $\widetilde{a}=\widetilde{b}=\Delta_1(\widetilde{a})$, and by Proposition \ref{p minimal partial isometries in arbitrary dimension}$(b)$ we have
$$\left\{ v\in \mathcal{U}_{min} (H) : f_{\widetilde{a}}(v) = \min_{e\in \mathcal{U}_{min} (H)}\!\! f_{\widetilde{a}} (e) \right\} = \left\{ v\in \mathcal{U}_{min} (H) : v\leq u \right\}$$ $$= \left\{ v\in \mathcal{U}_{min} (H) : f_{\widetilde{b}}(v) = \min_{e\in \mathcal{U}_{min} (H)}\!\! f_{\widetilde{b}} (e) \right\},$$ which implies that if $v$ is a minimal projection such that $v\leq u$ then $\Delta_1(v)$ is a minimal projection satisfying $\Delta_1(v)\leq u$.\smallskip

Let $e$ be an arbitrary minimal partial isometry in $B(H)$. Let us find another minimal partial isometry $v\in B(H)$ satisfying $e\perp v$. We consider the unitaries $u_1 = e + v$  and $u_2 = e -v$. Since $e\leq u_j$ for all $j=1,2$, it follows from the conclusion in the above paragraph that $\Delta_1 (e) \leq u_j$ for all $j=1,2$. It can be easily deduced from this and the minimality of $\Delta_1(e)$ that $\Delta_1 (e) = e$. We have therefore shown that $\Delta_1(e) =e$ for every minimal projection $e\in B(H)$.\smallskip

Finally, let $a$ be an element in $S(C_p(H))$. For each $e\in \mathcal{U}_{min} (H)$ we have $$\| \Delta_1(a) - e \|_p =\| \Delta_1(a) - \Delta_1(e) \|_p = \| a - e \|_p,$$ and consequently, an application of Proposition \ref{p a la Nagy} proves that $\Delta_1 (a) = a$, for every $a$ in $S(C_p(H))$, which finishes the proof.
\end{proof}

Let $e$ be a partial isometry in $B(H)$. It is known that $B(H)= B(H)_0(e) \oplus B(H)_1 (e)\oplus B(H)_2(e)$, where $B(H)_0(e)$, $B(H)_1(e)$ and $B(H)_2 (e)$ are the so-called \emph{Peirce subspaces} associated with $e$, which are defined by $$B(H)_2(e) = ee^* B(H) e^*e, B(H)_1(e) = (1-ee^*) B(H) e^*e \oplus ee^* B(H)(1-e^*e),$$ and $$B(H)_0(e) = (1-ee^*) B(H) (1-e^*e)$$. The natural projection of $B(H)$ onto $B(H)_j(e)$ is called the \emph{Peirce $j$-projection}, and it will be denoted by $P_j(e)$. When $e$ is a minimal partial isometry the Peirce subspace $B(H)_2 (e)$ coincides with $\mathbb{C} e$, and in such a case, for each $x\in B(H)$, we shall write $\varphi_e (x)$ for the unique complex number satisfying $P_2(e) (x) =\varphi_e(x) e$.\smallskip

Our next result is a first application of the previous Theorem \ref{t Tingley's for Cp with H 2dimensional}.

\begin{proposition}\label{p minimal partial isometries relative position} Let $H$ and $H'$ be complex Hilbert spaces, let $ p \in(1,\infty)\backslash\{2\}$, and let $\Delta: S(C_p(H))\to S(C_p(H'))$ be a surjective isometry. Suppose $e$ and $v$ are minimal partial isometries in $S(C_p(H))$. Then $P_2(\Delta(e)) (\Delta(v))$ belongs to the set $\{ \varphi_e (v ) \Delta(e) , \overline{\varphi_e (v )} \Delta(e)\},$ equivalently, $ \varphi_{\Delta(e)} (\Delta(v) ) =  \varphi_e (v ) $ or $ \varphi_{\Delta(e)} (\Delta(v) ) =  \overline{\varphi_e (v )},$
and $\| P_0 (\Delta(e)) (\Delta(v))\|_p = \|P_0(e) (v) \|_p$.
\end{proposition}

\begin{proof} By Remark \ref{remark on a couple of minimal partial isometries} we can find a family of mutually orthogonal minimal partial isometries $\{e_1,e_2\}\cup\{e_j:j\in J\}$ such that $\displaystyle \bigcap_{j\in J} \{ e_j\}^{\perp} \cong C_p(H_1),$ where $H_1$ is a two dimensional complex Hilbert space, $e= e_1$, and $v \in S\left(\displaystyle \bigcap_{j\in J} \{ e_j\}^{\perp}\right) \cong S(C_p(H_1))\subset S(C_p(H))$. By Lemma \ref{l preservation of orthogonality} and Proposition \ref{p minimal partial isometries in arbitrary dimension} we get $$\Delta\left(  \left(\displaystyle \bigcap_{j\in J} \{ e_j\}^{\perp}\right) \cap S(C_p(H)) \right) = \left(\displaystyle \bigcap_{j\in J} \{ \Delta(e_j)\}^{\perp}\right) \cap S(C_p(H'))\cong S(C_p(H_1)).$$ By restricting $\Delta$ to $\displaystyle \left(\bigcap_{j\in J} \{ e_j\}^{\perp}\right) \cap S(C_p(H))$ we can assume that $H=H'= H_1$ is a two dimensional complex Hilbert space, and $\Delta : S(C_p(\mathbb{C}^2))\to S(C_p(\mathbb{C}^2))$ is a surjective isometry.\smallskip

By Theorem \ref{t Tingley's for Cp with H 2dimensional} there exists a surjective real linear isometry $T: C_p(\mathbb{C}^2)\to C_p(\mathbb{C}^2)$ whose restriction to $S(C_p(\mathbb{C}^2))$ is $\Delta$. It is known that, in this case, there exist unitaries $u,v\in M_2(\mathbb{C})$ such that one of the following statements holds:\begin{enumerate}[$(a)$] \item $T(x) = u x v$, for every $x\in C_p(\mathbb{C}^2)$;
\item $T(x) = u x^t v$, for every $x\in C_p(\mathbb{C}^2)$;
\item $T(x) = u \overline{x} v$, for every $x\in C_p(\mathbb{C}^2)$;
\item $T(x) = u x^* v$, for every $x\in C_p(\mathbb{C}^2)$,
\end{enumerate} where $\overline{(x_{ij})} = (\overline{x_{ij}})$ (just combine Proposition \ref{p minimal partial isometries uniformly homogeneous} and \cite[Theorem 11.2.3]{FleJam08}). Under these circumstances, it is a routine exercise to check that the desired conclusions hold.
\end{proof}

Combining Propositions \ref{p minimal partial isometries uniformly homogeneous} and \ref{p minimal partial isometries relative position} we get the following corollary.

\begin{corollary}\label{c prop 2.8 linear or conjugate linear} Let $H$ and $H'$ be complex Hilbert spaces, let $ p \in(1,\infty)\backslash\{2\}$, and let $\Delta: S(C_p(H))\to S(C_p(H'))$ be a surjective isometry. Suppose $e$ and $v$ are two minimal partial isometries in $S(C_p(H))$.
Then one of the following statements hold:\begin{enumerate}[$(a)$]\item If $\Delta (\lambda w) = \lambda \Delta(w)$ for every $\lambda\in \mathbb{T}$ and every $w\in \mathcal{U}_{min} (H),$ then $$P_2(\Delta(e)) (\Delta(v))=  \varphi_e (v ) \Delta(e),$$ equivalently, $ \varphi_{\Delta(e)} (\Delta(v) ) =  \varphi_e (v )$;
\item If $\Delta (\lambda w) = \overline{\lambda} \Delta(w)$ for every $\lambda\in \mathbb{T}$ and every $w\in \mathcal{U}_{min} (H),$ then $$P_2(\Delta(e)) (\Delta(v))= \overline{\varphi_e (v )} \Delta(e),$$ equivalently, $\varphi_{\Delta(e)} (\Delta(v) ) =  \overline{\varphi_e (v )}$.
\end{enumerate}
\end{corollary}

\begin{proof} We shall only prove statement $(a)$, the proof of $(b)$ is analogous. We therefore assume that $\Delta (\lambda w) = \lambda \Delta(w)$ for every $\lambda\in \mathbb{T}$ and every $w\in \mathcal{U}_{min} (H).$ Proposition \ref{p minimal partial isometries relative position} implies that $ P_2(\Delta(e)) (\Delta(v))\in \{ \varphi_e (v ) \Delta(e) , \overline{\varphi_e (v )} \Delta(e)\}.$ If $P_2(\Delta(e)) (\Delta(v)) = {\varphi_e (v )} \Delta(e)$ there is nothing to prove. Suppose that $P_2(\Delta(e)) (\Delta(v)) = \overline{\varphi_e (v )} \Delta(e)\neq 0.$ By assumptions $\Delta(i v) = i \Delta(v)$, and by Proposition \ref{p minimal partial isometries relative position} we have $$i \overline{\varphi_e (v )} \Delta(e)= P_2(\Delta(e)) (\Delta(i v))\in \{ \varphi_e (i v ) \Delta(e) , \overline{\varphi_e (i v )} \Delta(e)\},$$ which proves that $\varphi_e (v )\in \mathbb{R}$, and hence $P_2(\Delta(e)) (\Delta(v)) = {\varphi_e (v )} \Delta(e).$

\end{proof}

We are now in position to reveal a connection with the celebrated Wigner theorem. Let Proj$(H)$ denote the lattice of all projections on
a Hilbert space $H$ equipped with the usual partial ordering, and let $\hbox{Proj}_1(H)$ stand for the set of minimal (rank-one) projections on $H$. We recall that a conjugate-linear norm preserving bijection on $H$ is called an \emph{antiunitary operator}. Wigner's  unitary-antiunitary theorem reads as follows:

\begin{theorem}\label{Wigner thm}\rm{(Wigner's theorem \cite{CasVitLahLevr1997})} If $F : \hbox{Proj}_1(H)\to \hbox{Proj}_1(H)$ is a bijective function which preserves the transition probabilities, that is, $$\hbox{tr} \left(F(p) F(q) \right)= \hbox{tr} (pq), (p, q \in \hbox{Proj}_1(H)),$$
then there is an either unitary or antiunitary operator $u$ on $H$ such that $F$ is of the form $F(p) = u p u,$ for all $p\in \hbox{Proj}_1(H)$. $\hfill\Box$
\end{theorem}

We refer to \cite{Mol99,Mol2000} and \cite{Mol2002} for recent generalizations of Wigner's theorem. We are interested in a concrete extension established by L. Moln{\'a}r in \cite{Mol2002}. In the just quoted paper, Moln{\'a}r replaces the set, $\hbox{Proj}_1(H),$ of minimal projections on $H$ with the strictly wider set, $\mathcal{U}_{min} (H),$ of minimal partial isometries in $B(H)$ and determines the bijections on $\mathcal{U}_{min} (H)$ preserving the transition probabilities.

\begin{theorem}\label{Wigner thm min pi}\rm{\cite[Theorem 2]{Mol2002}} Let $F : \mathcal{U}_{min} (H)\to \mathcal{U}_{min} (H)$ be a bijective function preserving the transition probabilities, that is, $$\hbox{tr} \left(F(e)^* F(v)\right) = \hbox{tr} (e^* v), \hbox{ for all } e, v \in \mathcal{U}_{min} (H),$$
then $F$ is of one of the following forms:
\begin{enumerate}[$(a)$]\item There exist unitaries $\widehat{u},\widehat{v}$ on $H$ such that $F(e) = \widehat{u} e \widehat{v},$ for all $e\in \mathcal{U}_{min} (H)$, that is, $F$ coincides with the restriction to $\mathcal{U}_{min} (H)$ of a complex linear bijection on $B(H)$ preserving triple products of the form $\{a,b,c\}= \frac12 ( a b^* c+ cb^* a)$;
\item There exist antiunitaries $\widehat{u},\widehat{v}$ on $H$ such that $F(e) = \widehat{u} e^* \widehat{v},$ for all $e\in \mathcal{U}_{min} (H),$ that is, $F$ coincides with the restriction to $\mathcal{U}_{min} (H)$ of a complex linear bijection on $B(H)$ preserving triple products of the form $\{a,b,c\}= \frac12 ( a b^* c+ cb^* a)$. $\hfill\Box$
\end{enumerate}
\end{theorem}

If  $\Delta: S(C_p(H))\to S(C_p(H'))$ is a surjective isometry, where $H$ and $H'$ are complex Hilbert spaces, we deduce from Lemma \ref{l preservation of orthogonality} that $H$ and $H'$ are isometrically isomorphic. We can therefore restrict our study to the case in which $H=H'$.\smallskip

We can now establish our main result.

\begin{theorem}\label{t Tingley's for Cp with 2<p} Let $H$ be a complex Hilbert space, let
$p\in (1, \infty)\backslash\{2\}$ be a fixed real number, and let $\Delta: S(C_p(H))\to S(C_p(H))$ be a surjective isometry.
Then there exists a complex linear or a conjugate linear surjective isometry $T:C_p(H)\to C_p(H)$ whose restriction
to $S(C_p(H))$ is $\Delta$.
\end{theorem}

\begin{proof} We deduce from Corollary \ref{c minimal partial isometries in arbitrary dimension} that the restricted mapping $F=\Delta|_{_{\mathcal{U}_{min} (H)}}: \mathcal{U}_{min} (H)\to \mathcal{U}_{min} (H)$ is a surjective isometry. Corollary \ref{p minimal partial isometries uniformly homogeneous} assures that one of the following statements holds:\begin{enumerate}[$(a)$]\item $\Delta (\lambda v) = \lambda \Delta(v)$ for every $\lambda\in \mathbb{T}$ and every $v\in \mathcal{U}_{min} (H);$
\item $\Delta (\lambda v) = \overline{\lambda} \Delta(v)$ for every $\lambda\in \mathbb{T}$ and every $v\in \mathcal{U}_{min} (H).$
\end{enumerate}

Let us assume that $(a)$ holds. Corollary \ref{c prop 2.8 linear or conjugate linear} tells that $P_2(\Delta(e)) (\Delta(v))=  \varphi_e (v ) \Delta(e),$ equivalently, $ \varphi_{\Delta(e)} (\Delta(v) ) =  \varphi_e (v )$, for every $e,v\in \mathcal{U}_{min} (H).$ It is a routine exercise to check that in this case $\hbox{tr} (\Delta(e)^* \Delta(v)) = \hbox{tr} (e^* v)$, for every $e,v\in \mathcal{U}_{min} (H).$ Moln{\'a}r's theorem (see Theorem \ref{Wigner thm min pi}) combined with our hypothesis assure the existence of two unitaries (respectively, two antiunitaries) $\widehat{u},\widehat{v}$ on $H$ such that \begin{equation}\label{eq coincidence on min pi 1502} \Delta(e) = \widehat{u}\ e\ \widehat{v}, \hbox{(respectively, } \Delta(e) = \widehat{u}\ e^*\ \widehat{v}{\rm)},
\end{equation} for all $e \in \mathcal{U}_{min} (H)$. We define a surjective isometry $\Delta_1 : S(C_p(H))\to S(C_p(H))$ given by $\Delta_1 (x) = \widehat{u}^*\ \Delta(x) \ \widehat{v}^*$ (respectively, $\Delta_1 (x) = \widehat{v}\ \Delta(x)^* \ \widehat{u}$). It follows from \eqref{eq coincidence on min pi 1502} that $\Delta_1(e) = e,$ for all $e \in \mathcal{U}_{min} (H)$.\smallskip

Fix a finite rank operator $a\in S( C_p(H)),$ and let us pick a family of mutually orthogonal minimal partial isometries $\{e_j:j\in J\}\subset S( C_p(H))$ such that $\displaystyle a\in S\left(\displaystyle \bigcap_{j\in J} \{ e_j\}^{\perp}\right) \cong S(C_p(H_1))\subset S(C_p(H)),$ where $H_1$ is a finite dimensional complex Hilbert space. By Lemma \ref{l preservation of orthogonality} and Proposition \ref{p minimal partial isometries in arbitrary dimension} we get $$\Delta_1\left(  \left(\displaystyle \bigcap_{j\in J} \{ e_j\}^{\perp}\right) \cap S(C_p(H)) \right) = \left(\displaystyle \bigcap_{j\in J} \{ \Delta_1(e_j)\}^{\perp}\right) \cap S(C_p(H))\cong S(C_p(H_1)).$$ By restricting $\Delta_1$ to $\displaystyle \left(\bigcap_{j\in J} \{ e_j\}^{\perp}\right) \cap S(C_p(H))$ we can assume that $H= H_1$ is a finite dimensional complex Hilbert space, and $\Delta_1 : (M_m(\mathbb{C}), \|.\|_p)\to (M_m(\mathbb{C}), \|.\|_p)$ is a surjective isometry for a suitable natural $m$.\smallskip

Under these assumptions we have $$\|\Delta_1(a)- e \|_p=\|\Delta_1(a)-\Delta_1(e) \|_p= \| a- e \|_p,$$ for every $e \in \mathcal{U}_{min} (H)$.  Proposition \ref{p a la Nagy} now implies that $\Delta_1 (a) = a$.\smallskip

We have shown that $\Delta_1(a) = a,$ for every finite rank operator $a\in S( C_p(H)).$ Since finite rank elements in $S( C_p(H))$ are norm dense in $S( C_p(H)),$ we conclude from the continuity of $\Delta_1$ that $\Delta_1 (x) = x,$ for every $x\in S( C_p(H)),$ and consequently $\Delta (x) =  \widehat{u}\ x\ \widehat{v},$ for every $x\in S( C_p(H)).$\smallskip

We assume next that statement $(b)$ holds. Let us take a conjugate linear $^*$-automorphism of period-2 $\overline{\ \cdot\ }$ on $B(H)$ whose restriction to $C_p(H)$ defines a conjugate linear isometry of period-2 on the latter space. The mapping $\Delta_2 : S(C_p(H))\to S(C_p(H))$ given by $\Delta_2 (x) =  \overline{\Delta(x)} $ is a surjective isometry satisfying $$\Delta_2 (\lambda v)= \overline{\Delta (\lambda v)}= \overline{\overline{\lambda} \Delta (v)} = \lambda \overline{\Delta(v)}=\lambda \Delta_2 ( v)$$ for every $\lambda\in \mathbb{T}$ and every $v\in \mathcal{U}_{min} (H)$. Therefore, by applying the conclusion in the previous part, we deduce the existence of unitaries or antiunitaries $\widehat{u},\widehat{v}$ on $H$ such that $\Delta_2 (x) =  \widehat{u} x \widehat{v},$ for every $x\in S( C_p(H)).$ Consequently, $$ \Delta (x) = \overline{\Delta_2 (x)} = \overline{\widehat{u}\ x\ \widehat{v}} = \overline{\widehat{u}}\  \overline{x}\  \overline{\widehat{v}},$$ for every $x\in S( C_p(H)).$
\end{proof}

\medskip

\textbf{Acknowledgements} First and third authors partially supported by the Spanish Ministry of Economy and Competitiveness (MINECO) and European Regional Development Fund project no. MTM2014-58984-P and Junta de Andaluc\'{\i}a grant FQM375. Second author partially supported by the Spanish Ministry of Economy and Competitiveness Project MTM2013-43540-P.\smallskip

The authors are grateful to the referee for a careful reading of the paper, an outstanding report and valuable suggestions and comments. Our conclusions have been established in full generality thanks to sharp comments by the Referee.


\begin{thebibliography}{0}






\bibitem{Bhat2007} R. Bhatia, \emph{Perturbation bounds for matrix eigenvalues}, Society for Industrial and Applied Mathematics,  Philadelphia, 2007.

\bibitem{BhatSemr96} R. Bhatia, P. \v{S}emrl, Distance between hermitian operators in Schatten classes, \emph{Proc. Edinb. Math. Soc.} \textbf{39}, (1996), 377-380.





\bibitem{CasVitLahLevr1997} G. Cassinelli, E. de Vito, P.J. Lahti, A. Levrero, Symmetry groups in quantum mechanics and the theorem of Wigner on the symmetry transformations, \emph{Rev. Math. Phys.} \textbf{9} (1997), 921-941.








\bibitem{Ding2002} G.G. Ding, The 1-Lipschitz mapping between the unit spheres of two Hilbert spaces can be extended to a real linear isometry of the whole space, \emph{Sci. China Ser. A} \textbf{45} (2002), no. 4, 479-483.

\bibitem{Di:p} G.G. Ding, The isometric extension problem in the spheres of $l^p (\Gamma)$ $(p>1)$ type spaces, \emph{Sci.\ China Ser.\ A} \textbf{46} (2003), 333--338.


\bibitem{Di:8} G.G. Ding, The representation theorem of onto isometric mappings between two unit spheres of $l^\infty$-type spaces and the application on isometric extension problem, \emph{Sci.\ China Ser.\ A} \textbf{47} (2004), 722--729.

\bibitem{Di:1} G.G. Ding, The representation theorem of onto isometric mappings between two unit spheres of $l^1 (\Gamma)$ type spaces and the application to the isometric extension problem, \emph{Acta.\ Math.\ Sin.\ (Engl.\ Ser.)} \textbf{20} (2004), 1089--1094.

\bibitem{Ding07} G.G. Ding, The isometric extension of the into mapping from a $\mathcal{L}^\infty(\Gamma )$-type space to some Banach space, \emph{Illinois J. Math.} \textbf{51} (2) (2007), 445-453.

\bibitem{Ding2009} G.G. Ding, On isometric extension problem between two unit spheres, \emph{Sci. China Ser. A} \textbf{52} (2009), 2069-2083.



\bibitem{DunSchw63} N. Dunford, J.T. Schwartz, \emph{Linear operators. Part II: Spectral theory. Self adjoint operators in Hilbert space}, Interscience Publishers John Wiley \& Sons, New York-London, 1963.





\bibitem{FangWang06} X.N. Fang, J.H. Wang, Extension of isometries between the unit spheres of normed space $E$ and $C(\Omega)$,
\emph{Acta Math. Sinica, Engl. Ser.}, \textbf{22} (2006), 1819-1824.

\bibitem{FerGarPeVill17} F.J. Fern\'andez-Polo, J.J. Garc{\'e}s, A.M. Peralta, I. Villanueva, Tingley's problem for spaces of trace class operators,  \emph{Linear Algebra Appl.} \textbf{529} (2017), 294-323.








\bibitem{FerPe17} F.J. Fern\'andez-Polo, A.M. Peralta, Low rank compact operators and Tingley's problem, preprint 2016. arXiv:1611.10218v1

\bibitem{FerPe17b} F.J. Fern\'andez-Polo, A.M. Peralta, On the extension of isometries between the unit spheres of a C$^*$-algebra and $B(H)$, to appear in \emph{Trans. Amer. Math. Soc.} \textbf{5} (2018), 63-80.

\bibitem{FerPe17c} F.J. Fern\'andez-Polo, A.M. Peralta, Tingley's problem through the facial structure of an atomic JBW$^*$-triple, \emph{J. Math. Anal. Appl.} \textbf{455} (2017), 750-760.

\bibitem{FerPe17d} F.J. Fern\'andez-Polo, A.M. Peralta, On the extension of isometries between the unit spheres of von Neumann algebras, preprint 2017. arXiv:1709.08529v1

\bibitem{FleJam08} R. Fleming, J. Jamison, \emph{Isometries on Banach Spaces: Vector-Valued Function Spaces}, vol. 2 Chapman \& Hall/CRC Monogr. Surv. Pure Appl. Math., vol. \textbf{138}, Chapman and Hall/CRC, Boca Raton, London, New York, Washington, DC (2008).


\bibitem{GohbergKrein} I.C. Gohberg, M.G. Krein, \emph{Introduction to the theory of linear nonselfadjoint operators}. Translations of Mathematical Monographs, Vol. 18 American Mathematical Society, Providence, R.I., 1969.






\bibitem{JVMorPeRa2017} A. Jim{\'e}nez-Vargas, A. Morales-Campoy, A.M. Peralta, M.I. Ramírez, The Mazur-Ulam property for the space of complex null sequences, to appear in \emph{Linear and Multilinear Algebra}. https://doi.org/10.1080/03081087.2018.1433625 arXiv:1708.08538v2







\bibitem{Liu2007} R. Liu, On extension of isometries between unit spheres of $\mathcal{L}^{\infty}(\Gamma)$-type space and a Banach space $E$, \emph{J. Math. Anal. Appl.} \textbf{333} (2007), 959-970.


\bibitem{McCarthy67} C.A. McCarthy, $C_p$, \emph{Israel J. Math.} \textbf{5} (1967), 249-271.



\bibitem{Mol99} L. Moln{\'a}r, A generalization of Wigner's unitary-antiunitary theorem to Hilbert modules, \emph{J. Math.
Phys.} \textbf{40} (1999), 5544-5554.

\bibitem{Mol2000} L. Moln{\'a}r, Generalization of Wigner's unitary-antiunitary theorem for indefinite inner product spaces, \emph{Comm. Math. Phys.} \textbf{201} (2000), 785-791.

\bibitem{Mol2002} L. Moln{\'a}r, On certain automorphisms of sets of partial isometries, \emph{Arch. Math.} \textbf{78} (2002), 43-50.

\bibitem{Mori2017} M. Mori, Tingley's problem through the facial structure of operator algebras, preprint 2017. arXiv:1712.09192v1

\bibitem{Nag2013} G. Nagy, Isometries on positive operators of unit norm, \emph{Publ. Math. Debrecen} \textbf{82} (2013), 183-192.



\bibitem{Pe2017} A.M. Peralta, Extending surjective isometries defined on the unit sphere of $\ell_\infty(\Gamma)$, preprint 2017. arXiv:1709.09584v1

\bibitem{Pe2018} A.M. Peralta, A survey on Tingley's problem for operator algebras, to appear in \emph{Acta Sci. Math. Szeged}. arXiv:1801.02473v1

\bibitem{PeTan16} A.M. Peralta, R. Tanaka, A solution to Tingley's problem for isometries between the unit spheres of compact C$^*$-algebras and JB$^*$-triples, to appear in \emph{Sci. China Math.} arXiv:1608.06327v1.


\bibitem{S} S. Sakai, C$^*$-algebras and $W^*$-algebras. Springer Verlag. Berlin (1971).


\bibitem{Tak} M. Takesaki, \newblock {\em Theory of operator algebras I}, \newblock Springer, New York, 2003.

\bibitem{Ta:8} D. Tan, Extension of isometries on unit sphere of $L^\infty$, \emph{Taiwanese J. Math.} \textbf{15} (2011), 819--827.

\bibitem{Ta:1} D. Tan, On extension of isometries on the unit spheres of $L^p$-spaces for $0<p \leq 1$, \emph{Nonlinear Anal.} \textbf{74} (2011), 6981-6987.

\bibitem{Ta:p} D. Tan, Extension of isometries on the unit sphere of $L^p$-spaces, \emph{Acta.\ Math.\ Sin.\ (Engl.\ Ser.)} \textbf{28} (2012), 1197--1208.




\bibitem{Tan2016} R. Tanaka, The solution of Tingley's problem for the operator norm unit sphere of complex $n \times n$ matrices, \emph{Linear Algebra Appl.} \textbf{494} (2016), 274-285.

\bibitem{Tan2017} R. Tanaka, Spherical isometries of finite dimensional $C^*$-algebras, \emph{J. Math. Anal. Appl.} \textbf{445} (2017), no. 1, 337-341.

\bibitem{Tan2017b} R. Tanaka, Tingley's problem on finite von Neumann algebras, \emph{J. Math. Anal. Appl.} \textbf{451} (2017), 319-326.

\bibitem{Ting1987} D. Tingley, Isometries of the unit sphere, \emph{Geom. Dedicata} \textbf{22} (1987), 371-378.

\bibitem{Wang} R.S. Wang, Isometries between the unit spheres of $C_0(\Omega)$ type spaces, \emph{Acta Math. Sci.} (English Ed.) \textbf{14} (1994), no. 1, 82-89.


\bibitem{YangZhao2014} X. Yang, X. Zhao, On the extension problems of isometric and nonexpansive mappings. In: \emph{Mathematics without boundaries}. Edited by Themistocles M. Rassias and Panos M. Pardalos. 725-748, Springer, New York, 2014.

\end{thebibliography}
\end{document}